\documentclass[preprint,11pt]{elsarticle}
\usepackage{amsfonts}
\usepackage{amsmath}
\usepackage{amssymb, amsmath}
\usepackage{mathrsfs}
\usepackage[applemac]{inputenc}
\newcommand{\beq}{\begin{equation}}
\newcommand{\eeq}{\end{equation}}
\newcommand{\ben}{\begin{eqnarray}}
\newcommand{\een}{\end{eqnarray}}
\newcommand{\beno}{\begin{eqnarray*}}
\newcommand{\eeno}{\end{eqnarray*}}
\usepackage{times, ulem, amsfonts}
\usepackage{mathpazo, amsmath, amssymb, amsthm,color}

\newcommand{\w}[1]{\langle {#1} \rangle}
\newcommand{\p}{\partial}
\newcommand{\na}{\nabla}
\newcommand{\D}{\Delta}

\newcommand\diverg{\mathop{\mbox{\rm div}}}
\theoremstyle{plain}

\numberwithin{equation}{section}
\newtheorem{theorem}{Theorem}[section]
\newtheorem{corollary}[theorem]{Corollary}
\newtheorem{lemma}[theorem]{Lemma}
\newtheorem{proposition}[theorem]{Proposition}
\theoremstyle{definition}
\newtheorem{definition}[theorem]{Definition}
\newtheorem{remark}{Remark}

\newcommand{\ta}{\theta}
\newcommand{\e}{\varepsilon}
\newcommand{\R}{{\mathbb R}}

\newcommand\Z{{\mathbb{Z}}}

\begin{document}
\bibliographystyle{plainmma}
\baselineskip=24pt

\begin{center}
{\Large\bf Global well-posedness for
2-D Boussinesq system with the temperature-dependent viscosity and supercritical dissipation}\\[2ex]

\footnote{$^{*}$
Email Address: pingxiaozhai@163.com (X. Zhai); \
 zmchen@szu.edu.cn (Z. Chen); \
  bqdong@ahu.edu.cn (B. Dong).}
Xiaoping Zhai, Boqing Dong \ and \  Zhimin Chen\\[2ex]
School  of Mathematics and Statistics, Shenzhen University,
\\[0.5ex]
 Shenzhen,  Guangdong 518060, P. R. China\\[2ex]

\end{center}

\bigskip
\centerline{\large\bf Abstract}
The present paper is dedicated to the global well-posedness issue for the Boussinesq system with the temperature-dependent viscosity in $\R^2.$
We
aim at extending the work by Abidi and  Zhang  ( Adv. Math. 2017 {(305)} 1202--1249 ) to a supercritical dissipation  for temperature.

\date{\today}


\bigskip
\leftskip 0 true cm
\rightskip 0 true cm

\section{\Large\bf Introduction and the main result}
In this paper, we mainly study the Cauchy problem  of the   Boussinesq system with the temperature-dependent viscosity  in  $\R^2 $:
\begin{eqnarray}\label{m}
\left\{\begin{aligned}
&\partial_t\theta+u\cdot\nabla\theta+\kappa |D|^\alpha \theta =0,\\
&\partial_t u+ u\cdot\nabla u-\diverg(2\mu(\ta)d( u))+\nabla\Pi=\theta e_2,\\
&\diverg u =0,\\
&(\theta,u)|_{t=0}=(\theta_0,u_0),
\end{aligned}\right.
\end{eqnarray}
where $u=u(x,t)=(u_1(x,t),u_2(x,t))$ denotes the  velocity vector field, $d(u)=(\na u+\na u^T)/2$ denotes the deformation matrix, $\Pi=\Pi(x,t)$ is the scalar  pressure,
the scalar function $\theta=\theta(x,t)$ is the temperature, ${e}_2$ is the unit vector in $\R^2$,
the thermal conductivity coefficient
$\kappa \ge0$, the kinematic viscous coefficient $\mu(\ta)$ is a smooth,
positive and non-decreasing function on $[0, \infty)$.    Furthermore, in all that follows, we
shall always denote $|D|^s$ to be the Fourier multiplier with symbol $|\xi|^s$.
In the whole paper, we also assume that $\kappa =1$
and
\begin{align}\label{mutiaojian}
0<1\le \mu(\ta),\quad  \mu(\cdot)\in W^{2,\infty}(\R^+), \quad \mu(0)=1.
\end{align}

The Boussinesq system arises from a zeroth order approximation to the coupling between Navier-Stokes equations and the thermodynamic equations. It can be used as a model to describe many geophysical phenomena \cite{Pedloski}.
If we consider the
 more general Boussinesq system
 with the temperature-dependent viscosity and thermal diffusivity to
take the  following form:
\begin{eqnarray}\label{mm}
\left\{\begin{aligned}
&\partial_t\theta+u\cdot\nabla\theta-\diverg(\kappa(\ta)\na \ta)
 =0,\\
&\partial_t u+ u\cdot\nabla u-\diverg(\mu(\ta)\na u)+\nabla\Pi=\theta e_d,\\
&\diverg u =0,\\
&(\theta,u)|_{t=0}=(\theta_0,u_0),
\end{aligned}\right.
\end{eqnarray}
the problem becomes much more complicated.
Lorca and Boldr in \cite{Lorca1999}  proved the global existence of strong solution for small data, and the global existence of weak solution and the local existence and uniqueness of strong solution for general data in \cite{Lorca1996}.
Recently,
Wang and Zhang in \cite{wangchao} mainly used De-Giorgi method and Harmonic analysis tools to get the global existence of smooth solutions in $\R^2.$ Sun and Zhang in \cite{sunyongzhong}
extended the result in \cite{wangchao} to the case of bounded domain. More precisely, the authors in \cite{sunyongzhong}
got the global existence of strong solution to the initial-boundary value problem of the 2-D Boussinesq system and 3-D infinite Prandtl number model with viscosity and thermal conductivity depending on the temperature.
Li and Xu in \cite{lidong} also generalized the result in
\cite{wangchao} to the inviscid case (that is $\mu(\ta)=0$ ). They
got the global strong solution
for arbitrarily large initial data in Sobolev spaces $H^s(\R^2), s>2.$
 Francesco in \cite{Francesco}
obtained  the global existence of weak solutions to the
system \eqref{mm} in $\R^d$, with viscosity dependent on temperature. The initial temperature in \cite{Francesco} is only supposed to be bounded, while the initial velocity belongs to some critical Besov Space, invariant to the scaling of this system.
Jiu and Liu in \cite{jiu} obtained the global well-posedness of anisotropic nonlinear Boussinesq equations with horizontal temperature-dependent viscosity and vertical thermal diffusivity in $\R^2$.
Using $\kappa |D| \theta$ instead of $\diverg(\kappa(\ta)\na \ta)$ in system \eqref{mm}, Abidi and Zhang in \cite{zhangping2017} got  the global solution in $\R^2$
provided the viscosity coefficient is sufficiently close to some positive constant in $L^\infty$ norm.

When
$\kappa(\ta)$ and $\mu(\ta)$
 are two positive constants which do not depend on the temperature,
Cannon and  DiBenedetto in \cite{Cannon} used the classical method to get the
 global solutions in $\R^2.$
Recently, more and more researchers (see \cite{Bessaih}, \cite{chae},   \cite{danchin2017}, \cite{hmidi2009}, \cite{hmidi2010}, \cite{jiu}, \cite{ju}, \cite{Larios}, \cite{wugang}, \cite{wujiahong2016}, \cite{wujiahong2014}) pay much more attentions to the following model:
\begin{eqnarray}\label{mmm}
\left\{\begin{aligned}
&\partial_t\theta+u\cdot\nabla\theta+\kappa|D|^\alpha \ta
 =0,\\
&\partial_t u+ u\cdot\nabla u+\mu|D|^\beta \ta+\nabla\Pi=\theta e_d,\\
&\diverg u =0,\\
&(\theta,u)|_{t=0}=(\theta_0,u_0).
\end{aligned}\right.
\end{eqnarray}
where $\mu\geq 0$, $\kappa\geq 0$, $0\leq \alpha \le 2$ and $0\leq \beta\le 2$ are real parameters.
The fractional diffusion operators considered here  in appear naturally in the study in hydrodynamics as well as anomalous diffusion in semiconductor growth.
Mathematically, the problem for global regularity of \eqref{mmm} is an interesting and a subtle one. Intuitively, the lower the values of $\alpha, \beta$ are, the harder it is to prove that solutions emanating from sufficiently smooth and localized data persist globally. In particular, the problem with  no dissipation (i.e. $\mu=\kappa=0$) remains open. This is very similar to the Euler equation in two and three spatial dimensions  and in fact numerous studies explore the possibility of finite time blow up.

Our goal here is to relax the dissipation needed in \cite{zhangping2017} for global well-posedness in $\R^2$.
 More precisely, we get the following theorem:

\begin{theorem}\label{maintheorem}
For any $2/3<\alpha\le 1$, 	${8}/{(3\alpha-2)}< p < {1}/{C\|\mu(\cdot)-1\|_{L^\infty}}$, $ {\alpha }/{(2\alpha-1)}<q<\min\big\{2, {4\alpha}/{3(2\alpha-1)}\big\} $	and $3-2\alpha<s_0<{4\alpha }/{q}-8\alpha +6$. Assume $\ta_0\in (L^q\cap
\dot{H}^{-s_0}\cap H^{\alpha/2})\cap B_{p,\infty}^{\alpha/2}(\R^2) $ and $u_0\in
B_{\infty,1}^{-1}\cap H^1(\R^2)$ be a solenoidal vector filed. Then there exists some sufficiently small $\e_0$ so that if we assume
\begin{align}\label{tiaojian}
\|\mu(\cdot)-1\|_{L^\infty(\R^+)}\le\e_0,
\end{align}
 \eqref{m} has a unique global solution $( u, \ta)$ so that
\begin{align}\label{zhengzexing1}
&u\in C([0,\infty); H^1)\cap {\widetilde{L}}^2(\R^+;\dot{B}_{2,\infty}^{3/2})\cap L^1_{loc}(\R^+;B_{\infty,1}^{1}), \quad\partial_t u \in L^2(\R^+;L^2),
\end{align}
\begin{align}\label{zhengzexing2}
&\ta\in C([0,\infty);L^q\cap
\dot{H}^{-s_0}\cap H^{\alpha/2})\cap L^\infty(\R^+;B_{p,\infty}^{\alpha/2})\cap L^2(\R^+;H^\alpha)\cap \widetilde{L}^1_{loc}(\R^+;B_{p,\infty}^{{3\alpha}/{2}}).
\end{align}

\end{theorem}
Moreover, we have
\begin{align}\label{xiaoying}
\|\ta(t)\|_{L^2}\le C
{E}_0\w{t}^{-{s_0}/{\alpha}},
\end{align}
with
\begin{align}\label{e0}
{E}_0=\mathcal{E}_0(1+\mathcal{E}_0),
\quad\mathcal{E}_0=\|\ta_0\|_{\dot{H}^{-s_0}}+\|\ta_0\|_{L^2}
+(\|u_0\|_{L^2}+\|\ta_0\|_{L^q})
(1+\|\ta_0\|_{L^q}).
\end{align}
\begin{remark}
The proof about this theorem shares the same ideas as the case $\alpha=1$ treated in \cite{zhangping2017} but with much more technical difficulties.
\end{remark}

The paper is organized as follows. In Section 2, we recall the Littlewood-Paley theory and give some useful lemmas. In Section 3, we take several steps to give  the key a priori estimates.
 In Section 4, we complete the proof of our main theorem.

Let us complete this section by describing the notations which will be used in the sequel.
\noindent$\mathbf{Notations:}$ Let $A$, $B$ be two operators, we denote $[A, B] = AB - BA$, the commutator
between $A$ and $B$. For $a\lesssim b$, we mean that there is a uniform constant $C$, which may
be different on different lines, such that $a \le C b$.
For $X$ a Banach space and $I$ an interval of $\mathbb{R}$, we denote by $C(I; X)$ the set of
continuous functions on $I$ with values in $X$. For $q \in [1, +\infty]$, the notation $L^q (I; X)$ stands for the set of measurable
functions on $I$ with values in $X$, such that $t \rightarrow \|f(t)\|_{ X }$ belongs to $L^q (I)$.
We always let $(d_j)_{j\in\mathbb{Z}}$
(resp. $(c_j)_{j\in\Z}$) be a
generic elements of $\ell^1(\mathbb{Z})$ (resp. $\ell^2(\mathbb{Z})$) so that $\sum_{j\in\mathbb{Z}}d_j=1$ (resp. $\sum_{j\in\mathbb{Z}}c_j^2=1$).

  \bigskip

\section{\Large\bf Preliminaries}
In this section, we recall some basic facts on Littlewood-Paley theory (see \cite{bcd} for instance).
Let $\chi,\varphi$ be two smooth radial functions valued in the interval [0,1],
the support of $\chi$ be the ball $\mathscr{B}=\{\xi\in\R^d: |\xi|\leq{4}/{3}\}$,
the support of $\varphi$ be the annulus $\mathscr{C}=\{\xi\in\R^d: {3}/{4}\leq|\xi|\leq{8}/{3}\}$, so that
\begin{align*}
\sum_{j \in \mathbb{Z}} \varphi(2^{-j}\xi)&=1,\quad \forall\,\xi\in\R^d \setminus \{0\},
\end{align*}
\begin{align*}
\chi(\xi)+\sum_{j \ge0} \varphi(2^{-j}\xi)&=1, \quad \forall\,\xi\in\R^d.
\end{align*}
Let $h=\mathcal {F}^{-1}\varphi$ and $\widetilde{h}=\mathcal {F}^{-1}\chi$, the inhomogeneous dyadic blocks ${\Delta}_j$ are defined as follows:
\begin{align*}
\mathrm{if} \quad &j=-1, \quad{\Delta}_j f={\Delta}_{-1} f=\int_{\R^d}\widetilde{h}(y)f(x-y)dy,\\
\mathrm{if} \quad &j\ge0, \ \ \quad{\Delta}_j f=2^{jd}\int_{\R^d}h(2^{j}y)f(x-y)dy, \quad \mathrm{if }\quad j\le -2, \quad {\Delta}_j f=0.
\end{align*}
 The inhomogeneous low-frequency cut-off operator $S_j$ is defined by
 $S_j f=\sum_{j'\le j-1}{\Delta}_{j'}f.$

\begin{definition}
Let $s\in\R$ and $1\leq p,r\leq\infty$. The inhomogeneous Besov space ${B}_{p,r}^{s}(\R^d)$
consists of all the distributions $u$ in $\mathscr{S}^{'}(\R^d)$ such that
$$
\|u\|_{{B}_{p,r}^{s}}\stackrel{\mathrm{def}}{=}\left\|\big(2^{js}\|{\Delta}_{j}u\|_{L^{p}}\big)_{j\ge-1}\right\|_{\ell^{r}}
<\infty.
$$
\end{definition}
\begin{remark}\label{equivalent defn}
Let $s\in\mathbb{R}$, $1\leq p,r\leq\infty$ and $u\in\mathscr{S}^{'}(\mathbb{R}^d)$. Then there exists a positive constant $C$ such that $u$ belongs to ${B}_{p,r}^{s}(\mathbb{R}^d)$ if and only if there exists $\{c_{j,r}\}_{j\ge-1}$ such that $c_{j,r}\geq0$, $\|c_{j,r}\|_{\ell^{r}}=1$ and
$$
\|{\Delta}_{j}u\|_{L^p}\leq Cc_{j,r}2^{-js}\|u\|_{{B}_{p,r}^{s}},\quad \forall j\ge-1.
$$
If $r=1$, we denote by $d_j\stackrel{\mathrm{def}}{=}c_{j,1}$.
\end{remark}

We also need to use Chemin-Lerner type Besov spaces introduced in (see \cite{bcd}).
\begin{definition}
Let $s\in\mathbb{R}$ and $0<T\leq +\infty$. We define
$$
\|u\|_{\widetilde{L}_{T}^{\sigma}({B}_{p,r}^{s})}\stackrel{\mathrm{def}}{=}
\left(
\sum_{j\ge-1}
2^{jrs}\left(\int_{0}^{T}\|\dot{\Delta}_{j}u(t)\|_{L^p}^{\sigma}dt\right)^{{r}/{\sigma}}
\right)^{{1}/{r}}
$$
for $p\in[1,\infty]$, $r, \sigma \in [1,\infty)$,
and with the standard modification for $r=\infty$ or $\sigma=\infty$.
\end{definition}
\begin{remark}\label{neicha}
It is easy to observe that for $0<s_1<s_2,$ $\theta\in[0,1]$, $p,r,\lambda,\lambda_1,\lambda_2\in[1,+\infty]$, we have the following interpolation inequality in the Chemin-Lerner space (see \cite{bcd}):
\begin{align*}
\|u\|_{\widetilde{L}^\lambda_{T}({B}_{p,r}^s)}\le\|u\|^\theta_{\widetilde{L}^{\lambda_1}_{T}({B}_{p,r}^{s_1})}
\|u\|^{1-\theta}_{\widetilde{L}^{\lambda_2}_{T}({B}_{p,r}^{s_2})}
\end{align*}
with ${1}/{\lambda}={\theta}/{\lambda_1}+{(1-\theta)}/{\lambda_2}$ and $s=\theta s_1+(1-\theta)s_2$.
\end{remark}
Let us emphasize that, according to the Minkowski inequality, we have
\begin{align*}
\|f\|_{\widetilde{L}^\lambda_{T}({B}_{p,r}^s)}\le\|f\|_{L^\lambda_{T}({B}_{p,r}^s)}\hspace{0.5cm} \mathrm{if }\hspace{0.2cm}  \lambda\le r,\hspace{0.5cm}
\|f\|_{\widetilde{L}^\lambda_{T}({B}_{p,r}^s)}\ge\|f\|_{L^\lambda_{T}({B}_{p,r}^s)},\hspace{0.5cm} \mathrm{if }\hspace{0.2cm}  \lambda\ge r.
\end{align*}

The following Bernstein's lemma will be repeatedly used throughout this paper.

\begin{lemma}\label{bernstein}
Let $\mathcal{B}$ be a ball and $\mathcal{C}$ a ring
 of $\mathbb{R}^d$. A constant $C$ exists so that for any positive real number $\lambda$, any
non-negative integer k, any smooth homogeneous function $\sigma$ of degree m, and any couple of real numbers $(a, b)$ with
$1\le a \le b$, there hold
\begin{align*}
&&\mathrm{Supp} \,\hat{u}\subset\lambda \mathcal{B}\Rightarrow\sup_{|\alpha|=k}\|\partial^{\alpha}u\|_{L^b}\le C^{k+1}\lambda^{k+d(1/a-1/b)}\|u\|_{L^a},\\
&&\mathrm{Supp} \,\hat{u}\subset\lambda \mathcal{C}\Rightarrow C^{-k-1}\lambda^k\|u\|_{L^a}\le\sup_{|\alpha|=k}\|\partial^{\alpha}u\|_{L^a}
\le C^{k+1}\lambda^{k}\|u\|_{L^a},\\
&&\mathrm{Supp} \,\hat{u}\subset\lambda \mathcal{C}\Rightarrow \|\sigma(D)u\|_{L^b}\le C_{\sigma,m}\lambda^{m+d(1/a-1/b)}\|u\|_{L^a}.
\end{align*}
\end{lemma}
\begin{lemma}(see \cite{KP})\label{morser}
Let $s>0$, $1\le p,r\le \infty,$
$f,g\in L^\infty\cap{B}_{p,r}^{s}(\R^d)$,
then
\begin{equation*}%\label{morserinequality}
\|fg\|_{{B}_{p,r}^{s}}\le C\left(\|f\|_{{B}_{p,r}^{s}}\|g\|_{L^\infty}+\|g\|_{{B}_{p,r}^{s}}\|f\|_{L^\infty}\right).
\end{equation*}
\end{lemma}

The action of smooth functions on the space ${B}_{p,r}^{s}(\R^d)$ can be stated as follows:

\begin{lemma}\label{fuhe} (see \cite{bcd})
Let $I$ be an open interval of $\R$ and $F:I\rightarrow \R.$ Let $s>0$ and $\sigma$ be the smallest integer such that $\sigma\ge s,$ and $(p, r) \in [1, \infty]^2$. Assume that $F(0) =0$ and that $F''$ belongs to $W^{\sigma,\infty}(I; \R).$ Let $u, v\in {B}_{p,r}^{s}(\R^d)\cap L^\infty(\R^d)$ have values in $J\subset I.$ There exists a constant $C=C(s, I, J, N)$such that
\begin{align*}%\label{}
\|F(u)\|_{{B}_{p,r}^{s}}\le C(1+\|u\|_{L^\infty})^{\sigma}\|F''\|_{W^{\sigma,\infty}(I )} \|u\|_{{B}_{p,r}^{s}}
\end{align*}
and
\begin{align*}%\label{}
\|F\circ u-&F\circ v\|_{{B}_{p,r}^{s}}\le C(1+\|u\|_{L^\infty}+\|v\|_{L^\infty})^{\sigma}\|F''\|_{W^{\sigma,\infty}(I )} \nonumber\\
&\times( \| u- v\|_{{B}_{p,r}^{s}}\sup_{\tau\in[0,1]}\|v+\tau(u-v)\|_{L^\infty} +\|u-v\|_{L^\infty}\sup_{\tau\in[0,1]}\|v+\tau(u-v)\|_{{B}_{p,r}^{s}} ).
\end{align*}

\end{lemma}

We shall also use the following commutator's lemma to prove our theorem:
\begin{lemma} \label{jiaohuanzi}(Lemma 2.100 in \cite{bcd}).
Let $\sigma\in\R, 1\le r\le \infty,$ $1\le p\le p_1\le\infty$  and $v$ be a vector field over $\R^d$.
 Assume that
\begin{align*}%\label{jiao1}
\sigma>-d\min\{\frac{1}{p_1},\frac{1}{p'}\}\quad \mathrm{or}\quad \sigma>-1-d\min\{\frac{1}{p_1},\frac{1}{p'}\} \quad {if}\quad \mathrm{div }v=0.
\end{align*}
Define $
R_j\stackrel{\mathrm{def}}{=}[v\cdot\nabla, \Delta_j]f$ (or $R_j\stackrel{\mathrm{def}}{=}\mathrm{div}[v, \Delta_j]f$, if $\mathrm{div} v=0$). There exists a constant C depending continuously on $p,p_1,\sigma,$ and $d$, such that
\begin{align*}%\label{jiao2}
\|(2^{j\sigma}\|R_j\|_{L^p})_j\|_{\ell^r}\le C\|\nabla v\|_{B_{p_1,\infty}^{\frac dp}\cap L^{\infty}}\|f\|_{B_{p,r}^{\sigma}},\quad {if}\quad \sigma <1+\frac{d}{p_1}.
\end{align*}
Further, if $\sigma>0$ (or $\sigma>-1, \  if \ \mathrm{div} v=0$) and $\frac{1}{p_2}=\frac{1}{p}-\frac{1}{p_1}$, then
\begin{align*}%\label{jiao3}
\|(2^{j\sigma}\|R_j\|_{L^p})_j\|_{\ell^r}\le C(\|\nabla v\|_{ L^{\infty}}\|f\|_{B_{p,r}^{\sigma}}+\|\nabla v\|_{B_{p_1,r}^{\sigma-1}}\|\nabla f\|_{L^{p_2}}).
\end{align*}
Especially, when $\sigma>1+\frac{d}{p_1}$ ( or $\sigma=1+\frac{d}{p_1}$ and $r=1$), the above inequality ensures that
\begin{align*}%\label{jiao3-1}
\|(2^{j\sigma}\|R_j\|_{L^p})_j\|_{\ell^r}\le C\|\nabla v\|_{B_{p_1,r}^{\sigma-1}}\|f\|_{B_{p,r}^{\sigma}}.
\end{align*}
In the limit case $\sigma=-\min(\frac{d}{p_1},\frac{d}{p'})$ [or $\sigma=-1-\min(\frac{d}{p_1},\frac{d}{p'})$ if $\mathrm{div}v=0$ ],  we have
\begin{align*}%\label{jiao4}
\sup_{j\ge-1}2^{j\sigma}\|R_j\|_{L^p}\le C \|\nabla v\|_{B_{p_1,1}^{\frac {d}{p_1}}}\|f\|_{B_{p,\infty}^{\sigma}}.
\end{align*}
\end{lemma}
We will also use the following  Osgood's Lemma:
\begin{lemma}\label{osgood}(see \cite{bcd})
 Let $g\ge 0$ be a measurable function, $\gamma$ be a locally integrable function and $\Lambda$ be a positive, continuous and nondecreasing function. $a$ be a positive real number and assume that $g$ satisfy the inequality
 $$g(t)\le a+\int_{t_0}^t \gamma(s)\Lambda(g(s))ds.$$
 If $a>0$, then we have
 \begin{align*}%\label{}
-\Omega(g(t))+\Omega(a)\le\int_{t_0}^t \gamma(s)ds,
\end{align*}
 where
 \begin{align*}%\label{}
\Omega(x)=\int_x^1\frac{dr}{\Lambda(r)}.
\end{align*}
 If $a=0$ and $\Lambda$ satisfies
 $$\int_0^1\frac{dr}{\Lambda(r)}=+\infty,$$
 then the function $g\equiv0$.
\end{lemma}
Finally, we give
the $L^p$ estimate for the transport (-diffusion) equation .
\begin{lemma}\label{shuyun}(see \cite{cc})
Let $u$ be a smooth divergence-free vector field in $\R^d (d
\ge 2)$ and $\ta$ be a smooth solution of the following
transport (-diffusion) equation
\begin{align*}
\partial_t\theta+u\cdot\nabla\theta+\kappa |D|^\alpha \theta =f, \quad\diverg u=0,\quad \ta_{t=0}\big|=\ta_0,\quad \alpha\in(0,2),
\end{align*}
with $\kappa \ge 0$. Then  for any  $t\in \R^+$ and $1\le p\le \infty,$ there holds:
$$\|\ta\|_{L^p}\le\|\ta_0\|_{L^p}+\int_0^t\|f\|_{L^p}\,d\tau.$$
\end{lemma}
\section{\Large\bf The key a priori estimates}
In this section, we will use several steps to give the  key a priori estimates. Firstly, we present the basic energy estimate for $\ta$ and $u.$ Secondly, we give the
 derivative and improved derivative energy estimates for $\ta$ and $u$ respectively. In the last step, we get
$\|u\|_{L^1_t(\dot B^{1}_{\infty,1})}$ and $\|\ta\|_{L^1_t(B_{p,\infty}^{3\alpha/2})}$.

\subsection{ \large\bf The basic energy estimate for $\ta$ and $u$}
In order to explain the index we will be used more essentially
in the following, we will generalize our's argument to a $d $ dimension. More precisely, we get the following {proposition}:

\begin{proposition}\label{shuaijian}
Let $(\ta,u)$ be a smooth enough solution of the system \eqref{m} on $[0,T^*).$ Assume $\ta_0\in (L^q\cap L^2\cap
\dot{H}^{-s_0})(\R^d) $ and $u_0\in
L^2(\R^d)$.
For any $0<\alpha\le 1$, $s_0\in	( {\alpha d}/{q}+ { (d+2)}/{2}- { \alpha(d+4)}/{2}, {{2\alpha d}/{q}-\alpha d-6\alpha+3+ {3d}/{2}})$
for some $q\in( {2\alpha d}/{(6\alpha+\alpha d-d-2)}, 2)$, then there holds
\begin{align}\label{chujiao}
\|\ta(t)\|_{L^2}\lesssim
{E}_0\w{t}^{-{s_0}/{\alpha}},\quad for \quad \forall t<T^\ast.
\end{align}
Especially, If $d=2, $  for any $2/3<\alpha\le 1$,  $s_0\in	(\alpha, {4\alpha }/{q}-8\alpha +6)$, $ {\alpha }/{(2\alpha-1)}<q<\min\big\{2, {4\alpha}/{(3(3\alpha-2))}\big\} $,
there hold
\eqref{chujiao}
and
\begin{align*}%\label{xiaoping28}
&\|u\|_{L^\infty_t(L^2)}+
\|\ta\|_{L^\infty_t(L^2)}+\|\nabla u\|_{L^2_t(L^2)}
+\| \ta\|_{L^2_t(\dot{H}^{\alpha/2})}
\le CE_0.
\end{align*}%
\end{proposition}
\begin{proof}
 The key part to prove this {proposition} is to derive
the decay of $\|\ta(t)\|_{L^2}$. We will follow Schonbek's strategy in \cite{Schonbek} (or Proposition 4.1 in \cite{zhangping2017}) to obtain this decay.

On one hand,
we get by taking a standard $L^2$  energy estimates to the  $u$ equation of \eqref{m}  that
\begin{align}\label{xiaoping1}
\frac12\frac{d}{dt}\|u(t)\|_{L^2}^2+\int_{\R^d}\mu(\ta)d(u):d(u)\,dx=
\int_{\R^d}\ta u_d\,dx.
\end{align}%
Thanks to the H$\mathrm{\ddot{o}}$lder inequality, interpolation inequality and Young inequality, we infer from
\eqref{mutiaojian} and \eqref{xiaoping1} that
\begin{align*}%\label{xiaoping3}
\frac12\frac{d}{dt}\|u(t)\|_{L^2}^2+\|\nabla u(t)\|_{L^2}^2
\lesssim&\|\ta\|_{L^q}\|u\|_{L^{\frac{q}{q-1}}}\nonumber\\
\lesssim&\|\ta\|_{L^q}\|u\|_{L^2}^{1-(\frac{1}{q}-\frac{1}{2})d}\|\nabla u\|_{L^2}^{(\frac{1}{q}-\frac{1}{2})d}\nonumber\\
\lesssim&\|\ta\|_{L^q}^{\frac{4q}{(4+d)q-2d}}
\|u\|_{L^2}^{\frac{2((d+2)q-2d)}{(d+4)q-2d}}+\frac12\|\nabla u\|_{L^2}^2.
\end{align*}%
Applying Osgood's Lemma \ref{osgood} to the above inequality gives
\begin{align*}%\label{xiaoping5}
&\|u(t)\|_{L^2}^2+\|\nabla u(t)\|_{L^2_t(L^2)}^2\nonumber\\
&\quad\lesssim\|u_0\|_{L^2}^2
+\Big(\int_0^t\|\ta(t')\|_{L^q}^{\frac{4q}{(d+4)q-2d}}  \,dt'\Big)^{\frac{(d+4)q-2d}{8q}}
\end{align*}
which implies
\begin{align*}%\label{xiaoping6}
&\|u\|_{L^\infty_t(L^2)}^2+\|\nabla u\|_{L^2_t(L^2)}^2
\lesssim\|u_0\|_{L^2}^2+
\|\ta\|_{L_t^{\frac{4q}{(d+4)q-2d}} (L^q)}^2.
\end{align*}
Then by virtue of Lemma \ref{shuyun}, we  have
\begin{align}\label{xiaoping7}
&\|u\|_{L^\infty_t(L^2)}+\|\nabla u\|_{L^2_t(L^2)}
\lesssim\|u_0\|_{L^2}+
\|\ta_0\|_{L^q}t^{\frac{(d+4)q-2d}{4q}} .
\end{align}

On the other hand, we get, by taking $L^2$ inner product of the temperature equation in \eqref{m} with $ \ta$, that
\begin{align}\label{xiaoping8}
\frac12\frac{d}{dt}\|\ta(t)\|_{L^2}^2+\|\ta(t)\|_{\dot{H}^{\alpha/2}}^2=0.
\end{align}

Motivated by Schonbek's strategy for the classical Navier-Stokes system in \cite{Schonbek}   (see also \cite{zhangping2017}), we split the phase-space $\R^d$ into two time-dependent regions $S(t)\triangleq \{\xi\in\R^d,|\xi| \le g(t)\}$ and $S^c(t)$, the complement of the set $S(t)$ in $\R^d$, for some $g(t) \thicksim \w{t}^{-1/\alpha}$
to be determined hereafter.
 A simple computation can help us  get
 from \eqref{xiaoping8}  that
\begin{align}\label{xiaoping9}
\frac{d}{dt}\|\ta(t)\|_{L^2}^2+2(g(t))^\alpha\|\ta(t)\|_{L^2}^2\le 2(g(t))^\alpha\int_{S(t)}|\widehat{\ta}(t,\xi)|^2\,d\xi.
\end{align}
We have to deal with the term on the right hand side of \eqref{xiaoping9}.
According to Duhamel's formula, one can deduce from the first equation of \eqref{m}  that
\begin{align}\label{xiaoping11}
\widehat{\ta}(t,\xi)
=e^{-t|\xi|^\alpha}\widehat{\ta}_0(\xi)-\int_0^te^{-(t-t')|\xi|^\alpha} \xi\cdot\mathcal{F}_x(\ta u)(t',\xi)\,dt'.
\end{align}
As $\ta_0\in \dot{H}^{-s_0}(\R^d) $, we have
\begin{align}\label{xiaoping11+9866}
\int_{S(t)}|e^{-t|\xi|^\alpha}\widehat{\ta}_0(\xi)|^2\,d\xi\lesssim
{(g(t))}^{2s_0}\|\ta_0\|_{\dot{H}^{-s_0}}^2.
\end{align}
On one hand,
 it follows from Young's inequality that
\begin{align}\label{xiaoping12}
\int_{S(t)} |\int_0^te^{-(t-t')|\xi|^\alpha} \xi\cdot\mathcal{F}_x(\ta u)(t',\xi)\,dt'|^2\,d\xi
\lesssim& (g(t))^{d+2}\Big\|\int_0^te^{-(t-t')|\xi|^\alpha} \mathcal{F}_x(\ta u)(t',\xi)\,dt'\Big\|_{L^\infty}^2\nonumber\\
\lesssim& (g(t))^{d+2}\left(\int_0^t\|\ta u\|_{L^1}\,dt'\right)^2\nonumber\\
\lesssim& (g(t))^{d+2}\|\ta\|_{L^\infty_t(L^q)}^2\|u\|_{L^1_t(L^{\frac{q}{q-1}})}^2.
\end{align}
On the other hand, we can infer
 from \eqref{xiaoping7}  that
\begin{align}\label{xiaoping13}
\|u\|_{L^1_t(L^{\frac{q}{q-1}})}
\lesssim&\w{t}^{\frac{(d+4)q-2d}{4q}}
\|u\|_{L^{\frac{4q}{(2-q)d}}_t(L^{\frac{q}{q-1}})}\nonumber\\
\lesssim&\w{t}^{\frac{(d+4)q-2d}{4q}}
\|u\|_{L_t^\infty(L^2)}^{\frac{(d+2)q-2d}{2q}}\|\nabla u\|_{L_t^2(L^2)}^{\frac{(2-q)d}{2q}}
\lesssim\w{t}^{\frac{(d+4)q-2d}{2q}}(\|u_0\|_{L^2}+\|\ta_0\|_{L^q}).
\end{align}
Plugging the estimate \eqref{xiaoping13} into \eqref{xiaoping12} gives
\begin{align*}%\label{xiaoping15}
\int_{S(t)} |\int_0^te^{-(t-t')|\xi|^\alpha} \xi\cdot\mathcal{F}_x(\ta u)(t',\xi)\,dt'|^2\,d\xi
\lesssim& (g(t))^{d+2}\Big\|\int_0^te^{-(t-t')|\xi|^\alpha} \mathcal{F}_x(\ta u)(t',\xi)\,dt'\Big\|_{L^\infty}^2\nonumber\\
\lesssim& (g(t))^{d+2}\left(\int_0^t\|\ta u\|_{L^1}\,dt'\right)^2\nonumber\\
\lesssim& (g(t))^{d+2}(\|u_0\|_{L^2}+\|\ta_0\|_{L^q})^2
\|\ta_0\|_{L^q}^2\w{t}^{\frac{(d+4)q-2d}{q}},
\end{align*}
from which and estimate \eqref{xiaoping11+9866}, we finally infer that
\begin{align}\label{xiaoping16}
\int_{S(t)}|\widehat{\ta}(t,\xi)|^2\,d\xi\lesssim
(g(t))^{2s_0}\|\ta_0\|_{\dot{H}^{-s_0}}^2
+(g(t))^{d+2}\w{t}^{\frac{(d+4)q-2d}{q}}(\|u_0\|_{L^2}+\|\ta_0\|_{L^q})^2
\|\ta_0\|_{L^q}^2.
\end{align}
Inserting the above estimate \eqref{xiaoping16}  into \eqref{xiaoping9}, choosing
 $g(t) \thicksim \w{t}^{-1/\alpha}$  and using the assumption that $s_0\ge  {\alpha d}/{q}+ { (d+2)}/{2}- { \alpha(d+4)}/{2}$, we obtain
\begin{align}\label{xiaoping17}
&\frac{d}{dt}\|\ta(t)\|_{L^2}^2+2(g(t))^\alpha\|\ta(t)\|_{L^2}^2
\nonumber\\
&\quad\le
(g(t))^{\alpha+2s_0}\|\ta_0\|_{\dot{H}^{-s_0}}^2+(g(t))^{\alpha+d+2}\w{t}^{\frac{(d+4)q-2d}{q}}
(\|u_0\|_{L^2}+\|\ta_0\|_{L^q})^2
\|\ta_0\|_{L^q}^2
\nonumber\\
&\quad\le
\w{t}^{-\frac{\alpha+d+2}{\alpha}}\w{t}^{\frac{(d+4)q-2d}{q}}
\mathcal{E}_0^2
\end{align}
with $\mathcal{E}_0\triangleq\|\ta_0\|_{\dot{H}^{-s_0}}+\|\ta_0\|_{L^2}
+(\|u_0\|_{L^2}+\|\ta_0\|_{L^q})
(1+\|\ta_0\|_{L^q})$.

Multiplying by $\exp\big(2\int_0^t(g(t'))^\alpha\,dt'\big)$ on both hand sides of \eqref{xiaoping17} leads to
\begin{align*}%\label{xiaoping18}
\frac{d}{dt}\Big(\|\ta(t)\|_{L^2}^2\exp\big(2\int_0^t(g(t'))^\alpha\,dt'\big)\Big)\le
C\mathcal{E}_0^2\w{t}^{-\frac{\alpha+d+2}{\alpha}+\frac{(d+4)q-2d}{q}}\exp\big(2\int_0^t(g(t'))^\alpha\,dt'\big).
\end{align*}
Let us choose $g(t) =(\beta\w{t})^{-1/\alpha}$
for $\beta> {d+2}/{(2\alpha)}-{((d+4)q-2d)}/{2q}$ in the above inequality to get
$$
\w{t}^{2\beta}\|\ta(t)\|_{L^2}^2\lesssim\|\ta_0\|_{L^2}^2+
\mathcal{E}_0^2\w{t}^{2\beta-\frac{d+2}{\alpha}+\frac{(d+4)q-2d}{q}},$$
which implies for any $t\in(0,T^\ast)$
\begin{align}\label{xiaoping18+98}
\|\ta(t)\|_{L^2}\lesssim
\mathcal{E}_0\w{t}^{-\frac{d+2}{2\alpha}+\frac{(d+4)q-2d}{2q}}.
\end{align}
Combining with estimates
\eqref{xiaoping1} and  \eqref{xiaoping18+98}, we get for any $ {2\alpha d}/{(6\alpha+\alpha d-d-2)}<q<2$ that
\begin{align}\label{xiaoping19}
\|u(t)\|_{L^2}\lesssim&\|u_0\|_{L^2}+
\|\ta\|_{L^1_t(L^2)}
\lesssim\|u_0\|_{L^2}+\mathcal{E}_0\w{t}^{1-\frac{d+2}{2\alpha}+\frac{(d+4)q-2d}{2q}}
\lesssim \mathcal{E}_0\w{t}^{1-\frac{d+2}{2\alpha}+\frac{(d+4)q-2d}{2q}}.
\end{align}
Thanks to \eqref{xiaoping18+98}, \eqref{xiaoping19}, we get, by a similar derivation of \eqref{xiaoping16}, that
\begin{align}\label{xiaoping21}
\int_{S(t)}|\widehat{\ta}(t,\xi)|^2\,d\xi
\lesssim&
\mathcal{E}_0^2\w{t}^{-{2s_0}/{\alpha}}+(g(t))^{d+2}\left(\int_0^t\|u(t')\|_{L^2}
\|\ta(t')\|_{L^2}\,dt'\right)^2\nonumber\\
\lesssim&
\mathcal{E}_0^2\w{t}^{-{2s_0}/{\alpha}}+\mathcal{E}_0^4(g(t))^{d+2}
\left(\int_0^t\w{t'}^{1-\frac{d+2}{\alpha}+\frac{(d+4)q-2d}{q}}\,dt'\right)^2\nonumber\\
\lesssim&
\mathcal{E}_0^2\w{t}^{-{2s_0}/{\alpha}}+\mathcal{E}_0^4\w{t}^{-\frac{d+2}{\alpha}}
\w{t}^{4-\frac{2(d+2)}{\alpha}+\frac{2(d+4)q-4d}{q}}\nonumber\\
\lesssim&
\mathcal{E}_0^2\w{t}^{-{2s_0}/{\alpha}}+\mathcal{E}_0^4\w{t}^{4-\frac{3(d+2)}{\alpha}+\frac{2(d+4)q-4d}{q}}
\nonumber\\\lesssim&\mathcal{E}_0^2(1+\mathcal{E}_0^2)\w{t}^{-{2s_0}/{\alpha}}
\end{align}
in which we have  let $s_0\le{{2\alpha d}/{q}-\alpha d-6\alpha+3+ {3d}/{2}}$.

Inserting the estimate \eqref{xiaoping21} into \eqref{xiaoping9} gives
\begin{align*}%\label{xiaoping22}
\frac{d}{dt}\|\ta(t)\|_{L^2}^2+2(g(t))^{\alpha}\|\ta(t)\|_{L^2}^2
\lesssim\mathcal{E}_0^2(1+\mathcal{E}_0^2)\w{t}^{-1-{2s_0}/{\alpha}}
\triangleq{E}_0^2\w{t}^{-1-{2s_0}/{\alpha}}.
\end{align*}%
Thus taking $g(t) =(\beta \w{t})^{-1/{\alpha}}
$ for $\beta>{s_0}/{\alpha}$ in the above inequality, we get, by using a similar derivation of \eqref{xiaoping18+98}, that
$$
\w{t}^{2\beta}\|\ta(t)\|_{L^2}^2\lesssim\|\ta_0\|_{L^2}^2+
{E}_0^2\w{t}^{2\beta-{2s_0}/{\alpha}}.
$$
Divided this inequality by $\w{t}^{2\beta} $ leads to
$$%\label{xiaoping23}
\|\ta(t)\|_{L^2}\lesssim
{E}_0\w{t}^{-{s_0}/{\alpha}},\quad for \quad \forall t<T^\ast.
$$

If $d=2, $  for any $2/3<\alpha\le 1$,  $\alpha<s_0< {4\alpha }/{q}-8\alpha +6$ for some $q$ satisfies $ {\alpha }/{(2\alpha-1)}<q<\min\big\{2, {4\alpha}/{(3(3\alpha-2))}\big\} $,
 we finally get that
\begin{align*}%\label{xiaoping24}
\|\ta\|_{L^1_t(L^2)}\le CE_0,
\end{align*}%
from which and \eqref{xiaoping1}, \eqref{xiaoping8}  we infer that
\begin{align*}%\label{xiaoping28}
&\|u\|_{L^\infty_t(L^2)}+
\|\ta\|_{L^\infty_t(L^2)}+\|\nabla u\|_{L^2_t(L^2)}
+\| \ta\|_{L^2_t(\dot{H}^{\alpha/2})}
\le\|u_0\|_{L^2}+C\|\ta\|_{L^1_t(L^2)}
\le CE_0.
\end{align*}%
 This completes the proof of  Proposition \ref{shuaijian}.
\end{proof}

\subsection{\large\bf The derivative energy estimates for $\ta$ and $u$}
In this subsection, we will follow the method in \cite{zhangping2017} to get the derivative energy estimates for $\ta$ and $u$ in $\R^2.$ The first important estimate is to get the energy inequality of \eqref{m}. In fact, when $d=2$, under the assumptions of Theorem \ref{maintheorem}, we can deduce from Proposition \ref{shuaijian} that
\begin{align}
&\|\ta(t)\|_{L^2}\le C
{E}_0\w{t}^{-s_0/\alpha},\label{529ee}\\
&\| u(t)\|_{L^2}^2+\| \ta(t)\|_{L^2}^2+
2\int_0^t(\| \na u\|_{L^2}^2+\| \ta\|_{\dot{H}^{\alpha/2}}^2)\,d\tau
\lesssim
{E}_0^2,\label{529e3}
\end{align}
where $E_0$ is given in \eqref{e0}.

%%%%%%%%%%%%%%%%%%%%%%%%%%%%%%%%%%%%%%%%%%%%%%%%%%%%%%%%%%%%%%%%%%%%%%%
%%%%%%%%%%%%%%%%%%%%%%%%%%%%%%%%%%%%%%%%%%%%%%%%%%%%%%%%%%%%%%%%%%%%%%
%%%%%%%%%%%%%%%%%%%%%%%%%%%%%%%%%%%%%%%%%%%%%%%%%%%%%%%%%%%%%%%%%%%%%%%
%%%%%%%%%%%%%%%%%%%%%%%%%%%%%%%%%%%%%%%%%%%%%%%%%%%%%%%%%%%%%%%%%%%%%%
%%%%%%%%%%%%%%%%%%%%%%%%%%%%%%%%%%%%%%%%%%%%%%%%%%%%%%%%%%%%%%%%%%%%%%5
%%%%%%%%%%%%%%%%%%%%%%%%%%%%%%%%%%%%%%%%%%%%%%%%%%%%%%%%%%%%%%%%%%%%%%%
%%%%%%%%%%%%%%%%%%%%%%%%%%%%%%%%%%%%%%%%%%%%%%%%%%%%%%%%%%%%%%%%%%%%%%
%%%%%%%%%%%%%%%%%%%%%%%%%%%%%%%%%%%%%%%%%%%%%%%%%%%%%%%%%%%%%%%%%%%%%%%
%%%%%%%
%%%%%%%%%%%%%%%%%%%%%%%%%%%%%%%%%%%%%%%%%%%%%%%%%%%%%%%%%%%%%%%%%%%%%%%
%%%%%%%%%%%%%%%%%%%%%%%%%%%%%%%%%%%%%%%%%%%%%%%%%%%%%%%%%%%%%%%%%%%%%%
%%%%%%%%%%%%%%%%%%%%%%%%%%%%%%%%%%%%%%%%%%%%%%%%%%%%%%%%%%%%%%%%%%%%%%%
%%%%%%%%%%%%%%%%%%%%%%%%%%%%%%%%%%%%%%%%%%%%%%%%%%%%%%%%%%%%%%%%%%%%%%
%%%%%%%%%%%%%%%%%%%%%%%%%%%%%%%%%%%%%%%%%%%%%%%%%%%%%%%%%%%%%%%%%%%%%%%
%%%%%%%%%%%%%%%%%%%%%%%%%%%%%%%%%%%%%%%%%%%%%%%%%%%%%%%%%%%%%%%%%%%%%%
In the following, we continue to
 prove the $\dot{H}^1$ energy estimates for $u$.
More precisely, we obtain the following proposition:
\begin{proposition}\label{pingping565}
Let $(\ta, u)$ be a smooth enough solution of \eqref{m} on $[0, T^\ast).$ Then under the assumptions of Theorem \ref{maintheorem}, for any $4/\alpha<p<{1}/{C\|\mu(\cdot)-1\|_{L^\infty}}$ and any $t <T^\ast$, we have
\begin{align}\label{529pingping23}
\|\nabla u\|_{L^\infty_t(L^2)}^2+\|\p_t u\|_{L^2_t(L^2)}^2
\le&C(\|\nabla u_0\|_{L^2}^2+\|\ta\|_{L^2_t(L^2)}^2)\nonumber\\
&\times\exp\Big\{C\int_0^t\big((1+\|
u\|_{L^2}^2)\|\nabla u\|_{L^2}^2+\big\||D|^\alpha \ta\big\|_{L^2}^2 \big)\,dt'\Big\}
.
\end{align}

\end{proposition}

\begin{proof}
This proposition can be obtained  similarly to  Lemma 4.3 in \cite{zhangping2017} and  we only need to make fully use of  the renormalized equation $\partial_t\mu(\theta) + u\cdot\nabla\mu(\theta) +\kappa |D|^\alpha \mu(\theta) =0$.
For simplicity, we omit the details here.
\end{proof}

From \eqref{529ee}, we can easily deduce that
$
\|\ta\|_{L^2_t(L^2)}^2\lesssim  E_0^2,
$
thus, combining with  \eqref{529e3}, in order to get the $\dot{H}^1$ energy estimates for $u$, we have to estimate $\|\ta\|_{{L}^2_t(\dot{H}^{\alpha})}^2$. In fact, we get the following proposition:
\begin{proposition}\label{pingping1}
Let $(\ta, u)$ be a smooth enough solution of \eqref{m} on $[0, T^\ast).$ Then under the assumptions of Theorem \ref{maintheorem}, for any  $t <T^\ast$, we have
\begin{align}\label{529tianjue+00}
\|\ta\|_{\widetilde{L}^\infty_t(\dot{H}^{\alpha/2})}^2+
\|\ta\|_{{L}^2_t(\dot{H}^{\alpha})}^2
\le&\|\ta_0\|_{\dot{H}^{\alpha/2}}^2
+C\|\nabla u\|_{L^2_t(L^2)}^2\Big(1+\|\ta_0\|_{L^\infty}^2+\|\ta_0\|_{L^2}^2
\nonumber\\
&\quad\quad\quad\quad\quad+\|\ta_0 \|_{L^\infty}^2\ln(e+\|\ta_0\|_{B_{p,\infty}^{\alpha/2}}+\|\ta_0\|_{L^\infty}\|\nabla u\|_{L^2_t(L^p)})  \Big).
\end{align}
\end{proposition}
\begin{proof}
We first deduce from the first equation of \eqref{m} and the following commutator's estimate which the proof  can be founded in \cite{hmidi2011}
 $$\|[\Delta_j, u\cdot\nabla]\ta\|_{L^p}\le C \|\nabla u\|_{L^p}\|\ta\|_{B^0_{\infty,\infty}},\quad \forall\  1\le p\le \infty \quad and \quad \forall\   j\ge -1,$$
 that
\begin{align}\label{xinger}
\|\Delta_j\ta(t)\|_{L^p}\le\|\Delta_j\ta_0\|_{L^p}e^{-ct2^{j\alpha}}
+C\int_0^te^{-c(t-t')2^{j\alpha}}\|\nabla u\|_{L^p}\|\ta\|_{L^\infty} \,dt',
\end{align}
from which and Lemma \ref{shuyun}, we have
\begin{align*}%\label{}
\|\Delta_j\ta\|_{L^\infty_t(L^p)}
+2^{{j\alpha}/{2}}\|\Delta_j\ta\|_{L^2_t(L^p)}
\le C(\|\Delta_j\ta_0\|_{L^p}+2^{-{j\alpha}/{2}}
\|\ta_0\|_{L^\infty}\|\nabla u\|_{L^2_t(L^p)}).
\end{align*}
Multiplying by $2^{{j\alpha}/{2}}$ on both hand sides of the above inequality and then taking $supremum$ about $j$ that
\begin{align}\label{529tianjue}
\|\ta\|_{\widetilde{L}^\infty_t(B_{p,\infty}^{\alpha/2})}+
\|\ta\|_{{\widetilde{L}}^2_t(B_{p,\infty}^{\alpha})}\le&
\|\ta_0\|_{B_{p,\infty}^{\alpha/2}}+\|\ta_0\|_{L^\infty}\|\nabla u\|_{L^2_t(L^p)}.
\end{align}

In the following, applying $\dot{\Delta}_j$ to the first equation of \eqref{m} and then taking the $L^2$ inner product of the resulting equation with $\dot{\Delta}_j\ta$ that
\begin{align}\label{simna}
\frac12\frac{d}{dt}\|\dot{\Delta}_j\ta\|_{L^2}^2
+\int_{\R^2}\dot{\Delta}_j(|D|^\alpha \theta)\cdot\dot{\Delta}_j\ta\,dx
=-\int_{\R^2}\dot{\Delta}_j(u\cdot\na \ta)\cdot\dot{\Delta}_j\ta\,dx.
\end{align}
The Bony's decomposition will be applied to estimate the term on the right hand side of \eqref{simna} that
\begin{align}\label{caocao}
u\cdot\na \ta=\dot{T}_u\na \ta+\dot{T}_{\na \ta}u+\dot{R}(u,\na \ta).
\end{align}
By Lemma \ref{bernstein}, we have
\begin{align}\label{simna2}
\|\dot{\Delta}_j(\dot{T}_{\na \ta}u)\|_{L^2}
\lesssim&\sum_{|j-j'|\le4}\|\dot{S}_{j'-1}\na \ta\|_{L^\infty}\|\dot{\Delta}_{j'}u\|_{L^2}\nonumber\\
\lesssim&\sum_{|j-j'|\le4}\|\dot{S}_{j'-1} \ta\|_{L^\infty}\|\dot{\Delta}_{j'}\na u\|_{L^2}
\lesssim c_{j}(t)\| \ta\|_{L^\infty}\|\na u\|_{L^2}.
\end{align}
Similarly, using the fact that $\diverg u=0$ implies
\begin{align}\label{simna3}
\|\dot{\Delta}_j(\dot{R}(u,\na \ta))\|_{L^2}
\lesssim&2^j\|\dot{\Delta}_j(\dot{R}(u, \ta))\|_{L^2}
\lesssim2^j\sum_{j'\ge j-3}\|\dot{\Delta}_{j'}u\|_{L^2}
\|\widetilde{\dot{\Delta}}_{j'} \ta\|_{L^\infty}\nonumber\\
\lesssim&2^j\sum_{j'\ge j-3}c_{j'}(t)2^{-j'}\|\na u\|_{L^2}
\| \ta\|_{L^\infty}
\lesssim c_{j}(t)\|\na u\|_{L^2}
\| \ta\|_{L^\infty}.
\end{align}
The last term in \eqref{caocao} will be estimated through the following commutator's argument:
\begin{align}\label{simna4}
\int_{\R^2}\dot{\Delta}_j(\dot{T}_u\na \ta)\cdot\dot{\Delta}_j\ta\,dx=&\int_{\R^2}\sum_{|j-j'|\le5}(\dot{S}_{j'-1}u-\dot{S}_{j-1}u)\dot{\Delta}_j\dot{\Delta}_{j'}\nabla \ta\,dx\nonumber\\
&+\int_{\R^2}\sum_{|j-j'|\le5}[\dot{\Delta}_j,\dot{S}_{j'-1}u]\dot{\Delta}_{j'}\nabla \ta\cdot\dot{\Delta}_j\ta\,dx\nonumber\\
\triangleq& I_1+I_2.
\end{align}
Thanks to  Lemma \ref{bernstein} and the commutator's estimate in \cite{bcd}, we obtain
\begin{align}\label{simna5}
I_1\lesssim&2^{-j\alpha}c_{j}^2(t)\|\na u\|_{L^2}\|\ta\|_{\dot{H}^{\alpha}}\| \ta\|_{L^\infty}.
\nonumber\\
I_2\lesssim&2^{-j}\sum_{|j-j'|\le5}\|\dot{S}_{j'-1}\na u\|_{L^2}
\| \dot{\Delta}_{j'}\na \ta\|_{L^\infty}\|\dot{\Delta}_{j}\ta\|_{L^2}
\lesssim2^{-j\alpha}c_{j}^2(t)\|\na u\|_{L^2}\|\ta\|_{\dot{H}^{\alpha}}
\| \ta\|_{\dot{B}^0_{\infty,2}}.
\end{align}

Inserting the estimates  about \eqref{simna2}, \eqref{simna3}, \eqref{simna5} into \eqref{simna} and summing up about $j$ give
\begin{align*}%\label{529tianjue+1}
\|\ta\|_{\widetilde{L}^\infty_t(\dot{H}^{\alpha/2})}^2+
\|\ta\|_{{L}^2_t(\dot{H}^{\alpha})}^2
\le&\|\ta_0\|_{\dot{H}^{\alpha/2}}^2
+\int_0^t\|\ta\|_{\dot{H}^{\alpha}}(\| \ta\|_{L^\infty}\|\na u\|_{L^2}+\|\na u\|_{L^2}
\| \ta\|_{\dot{B}^0_{\infty,2}})\,dt'\nonumber\\
\le&\|\ta_0\|_{\dot{H}^{\alpha/2}}^2
+\varepsilon\|\ta\|_{{L}^2_t(\dot{H}^{\alpha})}^2+C(\|\ta\|_{L^\infty_t(L^\infty)}^2
+\|\ta \|_{L^\infty_t(\dot{B}^0_{\infty,2})}^2)\|\nabla u\|_{L^2_t(L^2)}^2.
\end{align*}
Choosing $\varepsilon$ small enough in the above inequality implies
\begin{align}\label{529tianjue+1}
\|\ta\|_{\widetilde{L}^\infty_t(\dot{H}^{\alpha/2})}^2+
\|\ta\|_{{L}^2_t(\dot{H}^{\alpha})}^2
\le&\|\ta_0\|_{\dot{H}^{\alpha/2}}^2
+C(\|\ta\|_{L^\infty_t(L^\infty)}^2
+\|\ta \|_{L^\infty_t(\dot{B}^0_{\infty,2})}^2)\|\nabla u\|_{L^2_t(L^2)}^2.
\end{align}
Note that for any positive integer $N$ and $p>4/\alpha$, we have
\begin{align*}%\label{}
\|\ta \|_{L^\infty_t(\dot{B}^0_{\infty,2})}
\le&\|\ta \|_{L^\infty_t(L^2)}+\Big(\sum_{1\le j\le N}\|\Delta_j\ta\|^2_{L^\infty_t(L^\infty)} \Big)^{1/2}
+\Big(\sum_{ j> N}\|\Delta_j\ta\|^2_{L^\infty_t(L^\infty)} \Big)^{1/2}\nonumber\\
\le&\|\ta_0 \|_{L^2}+\|\ta_0 \|_{L^\infty}N^{1/2}+2^{(2/p-\alpha/2)N}\|\ta\|_{\widetilde{L}^\infty_t(B_{p,\infty}^{\alpha/2})}.
\end{align*}
Choosing $N$ in the above inequality such that
$
2^{(\alpha/2-2/p)N}\sim\|\ta\|_{\widetilde{L}^\infty_t(B_{p,\infty}^{\alpha/2})}
,$
 we  have
\begin{align}\label{529tianjue1+1}
\|\ta \|_{L^\infty_t(\dot{B}^0_{\infty,2})}
\le C\Big(1+\|\ta_0 \|_{L^2}+\|\ta_0 \|_{L^\infty}
\ln^{\frac12}(e+\|\ta\|_{\widetilde{L}^\infty_t(B_{p,\infty}^{\alpha/2})})\Big)
.
\end{align}
Taking estimate \eqref{529tianjue} into the above estimate \eqref{529tianjue1+1} and then inserting the resulting inequality into \eqref{529tianjue+1}  give \eqref{529tianjue+00}.
\end{proof}
Inserting the estimate  about $\|\ta\|_{{L}^2_t(\dot{H}^{\alpha})}^2$ in Proposition \ref{pingping1} into \eqref{529pingping23}, one can deduce from \eqref{529e3} and estimate $\|\ta\|_{L^2_t(L^2)}^2\lesssim  E_0^2$ that
\begin{align}\label{529pingping23+123456}
\|\nabla u\|_{L^\infty_t(L^2)}^2+\|\p_t u\|_{L^2_t(L^2)}^2
\le&C(\|\nabla u_0\|_{L^2}^2+\|\ta\|_{L^2_t(L^2)}^2)\exp\Big\{C(1+\|
u\|_{L^\infty_t(L^2)}^2)\|\nabla u\|_{L^2_t(L^2)}^2 \Big\}\nonumber\\
&\times\exp\Big\{C
\|\ta_0\|_{\dot{H}^{\alpha/2}}^2
+C\|\nabla u\|_{L^2_t(L^2)}^2\Big(1+\|\ta_0\|_{L^\infty}^2+\|\ta_0\|_{L^2}^2
\nonumber\\
&\quad+\|\ta_0 \|_{L^\infty}^2\ln(e+\|\ta_0\|_{B_{p,\infty}^{\alpha/2}}+\|\ta_0\|_{L^\infty}\|\nabla u\|_{L^2_t(L^p)})  \Big)\Big\}\nonumber\\
\le&C(\|\nabla u_0\|_{L^2}^2+E_0^2)\exp\Big\{C(1+E_0^2)E_0^2\Big\}\nonumber\\
&\times\exp\Big\{C
\|\ta_0\|_{\dot{H}^{\alpha/2}}^2
+CE_0^2\Big(1+\|\ta_0\|_{L^\infty}^2+\|\ta_0\|_{L^2}^2\Big) \Big\}
\nonumber\\
&\times(e+\|\ta_0\|_{B_{p,\infty}^{\alpha/2}}+\|\ta_0\|_{L^\infty}\|\nabla u\|_{L^2_t(L^p)})^{C\|\ta_0 \|_{L^\infty}^2\|\nabla u\|_{L^2_t(L^2)}^2}\nonumber\\
\triangleq& CM_0\Big\{e+\|\ta_0\|_{B_{p,\infty}^{\alpha/2}}+\|\ta_0\|_{L^\infty}\|\nabla u\|_{L^2_t(L^p)}\Big\}^{C\|\ta_0 \|_{L^\infty}^2\|\nabla u\|_{L^2_t(L^2)}^2} .
\end{align}
In the following, we have to estimate $\|\nabla u\|_{L^2_t(L^p)}$.

Thanks to the fact
\begin{equation*}%\label{pingping13}
 \na
u=\na(-\D)^{-1}\mathbb{P}\diverg\bigl((\mu(\ta)-1)\na u\bigr)
-\na(-\D)^{-1}\mathbb{P}\diverg\bigl(\mu(\ta)\na u\bigr), \end{equation*}
and the interpolation inequality
\begin{equation*}
\|f\|_{L^{r}(\R^2)}\leq
C\sqrt{r}\|f\|_{L^2(\R^2)}^{2/r}\|\na
f\|_{L^2(\R^2)}^{1-2/r},\qquad 2\leq r<\infty, \end{equation*}
we can  deduce for any $p\in[2,\infty)$ that
$$
\|\nabla u\|_{L^p} \le C_0 \sqrt{p}\|\mu(\ta_0)-1\|_{L^\infty}\|\nabla
u\|_{L^p} +C\sqrt{p}\|\nabla u\|_{L^2}^{{2}/{p}}
\|\mathbb{P}\diverg\bigl(\mu(\ta)\na u\bigr)\|_{L^2}^{1-{2}/{p}}
$$
with $C_0>0$ being a universal constant.

\noindent Using the second equation of \eqref{m} and
taking $\e_0$ sufficiently
small in \eqref{tiaojian}, we obtain for $2\leq p\leq
{1}/({2C_0\|\mu(\ta_0)-1\|_{L^\infty}})$ that \begin{equation}\label{pingping15}
\begin{aligned}
\|\nabla u\|_{L^p} &\le C\sqrt{p}\|\nabla u\|_{L^2}^{{2}/{p}}
\|\partial_tu+u\cdot\nabla u-\ta e_2\|_{L^2}^{1-{2}/{p}}
\\&
\le C\sqrt{p}\|\nabla u\|_{L^2}^{{2}/{p}}
\big(\|\partial_tu\|_{L^2}^{1-{2}/{p}}
+\|u\|_{L^4}^{1-{2}/{p}}\|\nabla u\|_{L^4}^{1-{2}/{p}}+\|\ta\|_{L^2}^{1-{2}/{p}}\big).
\end{aligned}
\end{equation}
%%%%%%%%%%%%%%%%%%%%%%%%%%%%%%%%%%%%%%%%%%%%%%
Especially, taking $p=4$ in the above inequality \eqref{pingping15}, one has
\begin{align}\label{pingping31}
 \|\nabla u\|_{L^2_t(L^4)}
 \leq& C\bigl( \|\nabla
u\|_{L^2_t(L^2)}^{1/2}\|\partial_tu\|_{L^2_t(L^2)}^{1/2}\nonumber\\
&
+\|u\|_{L^\infty_t(L^2)}^{1/2}\|\nabla
u\|_{L^\infty_t(L^2)}^{1/2}\|\nabla
u\|_{L^2_t(L^2)}+\|\nabla
u\|_{L^\infty_t(L^2)}^{1/2}\|\ta\|_{L^2_t(L^2)}^{1/2}\bigr)\nonumber\\
 \leq& C(1+E_0)E_0^{1/2}(1+\|\partial_tu\|_{L^2_t(L^2)}^{1/2}+\|\nabla
u\|_{L^\infty_t(L^2)}^{1/2}),
\end{align}
from which and \eqref{pingping15}, we infer
\begin{align}\label{pingping32}
\|\nabla u\|_{L^2_t(L^p)}
\le& C\sqrt{p}\|\nabla u\|_{L^2_t(L^2)}^{2/p}
\big(\|\partial_tu\|_{L^2_t(L^2)}^{1-2/p}+\|u\|_{L^\infty_t(L^4)}^{1-2/p}\|\nabla u\|_{L^2_t(L^4)}^{1-2/p}+\|\ta\|_{L^2_t(L^2)}^{1-2/p}\big)\nonumber\\
\le& C\sqrt{p}E_0(1+E_0)^{1-2/p}(1+\|\partial_tu\|_{L^2_t(L^2)}
+\|\na u\|_{L^\infty_t(L^2)}).
\end{align}
Substituting the above inequality into \eqref{529pingping23} gives
\begin{align}\label{529pingping23+1}
\|\nabla u\|_{L^\infty_t(L^2)}^2+&\|\p_t u\|_{L^2_t(L^2)}^2
\le CM_0\Big\{e+\|\ta_0\|_{B_{p,\infty}^{\alpha/2}}\nonumber\\
&+C\sqrt{p}E_0(1+E_0)^{1-2/p}\|\ta_0\|_{L^\infty}(1+\|\partial_tu\|_{L^2_t(L^2)}
+\|\na u\|_{L^\infty_t(L^2)})\Big\}^{C\|\ta_0 \|_{L^\infty}^2\|\nabla u\|_{L^2_t(L^2)}^2},
\end{align}
where $M_0$ is defined in \eqref{529pingping23+123456}.

To close the $\dot{H}^1$ energy estimate about $u$, we also follow the method in \cite{zhangping2017} to prove the non-concentration of energy in the time variable. More precisely, we need the following lemma:

\begin{lemma}(see \cite{zhangping2017})\label{jiarui1}
Let $(\ta, u)$ be a smooth enough solution of \eqref{m} on $[0, T^\ast).$ Then under the assumptions of Theorem \ref{maintheorem}, for any  $t <T^\ast$, we have
\begin{align*}%\label{jiarui2}
\| u\|_{\widetilde{L}^\infty_t(L^2)}\le C E_0(1+E_0)
\end{align*}
for $E_0$ given by \eqref{e0}. If moreover, there holds \eqref{tiaojian}, then for any small enough constant $\nu>0$, there exists $\lambda>0$ such that if $0 \le t_1<t_2<T^\ast$ and $t_2-t_1\le\lambda$, there holds
\begin{align*}%\label{jiarui3}
\|\nabla u\|_{L^2([t_1,t_2];L^2)}\le \nu.
\end{align*}
\end{lemma}
With estimate \eqref{529pingping23+1} and Lemma \ref{jiarui1} in hand, we can also use the same boot-strap argument to get the global in time estimate of $\|\nabla u\|_{L^\infty_t(L^2)}$ and $\|\nabla u\|_{L^2_t(L^p)}$. The whole process can be obtained similarly to Proposition 4.2 in \cite{zhangping2017} without any difficulties. Here, we omit the details for convenience. Yet, we still use the same
 notations as in \cite{zhangping2017} in our further estimates.
In fact, we obtain the following proposition:

\begin{proposition}\label{jiarui13}
Let $(\ta, u)$ be a smooth solution of \eqref{m} on $[0, T^\ast).$ Then under the assumptions of Theorem \ref{maintheorem} and for some sufficiently small $\e$, for any  $t <T^\ast$, we have
\begin{align}
&\|\ta\|_{L^\infty_t(\dot{H}^{\alpha/2})}^2
+\ln(1+\|\ta\|^2_{L^\infty_t(B_{p,\infty}^{\alpha/2})})+\|\ta\|_{L^2_t(\dot{H}^{\alpha})}^2
\nonumber\\
&\quad\le C^{2+C\|\ta_0\|_{L^\infty}^2E^2_0}
\Big(\mathcal{A}+\mathcal{B}+\|\ta_0\|_{\dot{H}^{\alpha/2}}^2
+\ln\Big(1
+\|\ta_0\|_{B_{p,\infty}^{\alpha/2}}^2+\|\nabla u_0\|_{L^2}^2\Big)\Big)\triangleq \mathcal{G}_1,\label{jiarui14}
\\
&\|\nabla u\|_{L^\infty_t(L^2)}^2+\|\p_t u\|_{L^2_t(L^2)}^2
\le C(\|\nabla u_0\|_{L^2}^2+E_0^2)\exp(CE_0^2(1+E_0^2)+\mathcal{G}_1)\triangleq\mathcal{G}_2,\label{jiarui15}
\\
&
\|\nabla u\|_{L^2_t(L^p)}
\le C\sqrt{p}E_0(1+E_0)^{1-\frac{2}{p}}(1+\sqrt{\mathcal{G}_2})\triangleq\mathcal{G}_3,\label{jiarui16}
\end{align}
where
$
\mathcal{A}\triangleq CE_0^2(1+E_0^2+\|\ta_0\|_{L^\infty\cap L^2}^2),\quad
\mathcal{B}\triangleq\mathcal{A}+E_0^2\|\ta_0\|_{L^\infty}^2\ln(1+CE_0(1+E_0)\|\ta_0\|_{L^\infty}).
$
\end{proposition}

\bigskip

\subsection{ \large\bf The improved derivative energy estimates for $\ta$ and $u$}
With the $\dot{H}^1$ energy estimates for $u$ and $\dot{H}^{\alpha/2}$ for $\ta$ in hand in the last subsection, the  most important thing in what follows is to get
$\|u\|_{L^1_t(\dot B^{1}_{\infty,1})}$ and $\|\ta\|_{L^1_t(B_{p,\infty}^{3\alpha/2})}$.
More precisely, we get the following {proposition}:
\begin{proposition}\label{xiaoxiao1}
Let $(\ta, u)$ be a smooth enough solution of \eqref{m} on $[0, T^\ast).$ Then under the assumptions of Theorem \ref{maintheorem} and for any  $t <T^\ast$, if  we assume moreover that $u_0\in\dot B^{-1}_{\infty,1}$, there holds
\begin{align}\label{peinam1+896}
\|u\|_{L^1_t(\dot B^1_{\infty,1})}
\lesssim&\|u_0\|_{\dot B^{-1}_{\infty,1}}+t\|\ta_0\|_{L^q}^{ q/2}\|\ta_0\|_{L^\infty}^{1-q/2}
+E_0^2(1+E_0)+\mathcal{G}_3(t^{1/2}\|\ta_0\|_{L^2})^{\frac{\alpha p-4}{(\alpha+1)p-2}} \mathcal{G}_4^{\frac{ p+2}{(\alpha+1)p-2}}\nonumber\\
\triangleq&\mathcal{H}(t),
\end{align}
where $\mathcal{G}_4\triangleq\|\ta_0\|_{B_{p,\infty}^{\alpha/2}}+C\|\ta_0\|_{L^\infty}\mathcal{G}_3.$
\end{proposition}
\begin{proof}
Using \eqref{jiarui16} we  can obtain similarly to the first estimate in Proposition  \ref{pingping1} that
\begin{align}\label{pingping2+896}
\|\ta\|_{\widetilde{L}^\infty_t(B_{p,\infty}^{\alpha/2})}+
\|\ta\|_{\widetilde{L}^2_t(B_{p,\infty}^{\alpha})}\le
\|\ta_0\|_{B_{p,\infty}^{\alpha/2}}+\|\ta_0\|_{L^\infty}\|\nabla u\|_{L^2_t(L^p)}\le\|\ta_0\|_{B_{p,\infty}^{\alpha/2}}
+C\|\ta_0\|_{L^\infty}\mathcal{G}_3\triangleq\mathcal{G}_4.
\end{align}
From equation \eqref{m}, one can easily deduce that
\begin{align}\label{jiarui5+258}
\dot{\Delta}_ju(t)=e^{t\Delta}\dot{\Delta}_ju_0
+\int_0^te^{-(t-t')\Delta}\dot{\Delta}_j\mathbb{P}\Big\{
\diverg(2(\mu(\ta)-1)d (u))-u\cdot\nabla u+\theta e_2\Big\}
(t')\,dt'.
\end{align}
A standard energy estimate gives
\begin{align}\label{peinam1}
\|u\|_{L^1_t(\dot B^1_{\infty,1})}
\lesssim&\|u_0\|_{\dot B^{-1}_{\infty,1}}+\|u\cdot\na u\|_{L^1_t(\dot B^{-1}_{\infty,1})}+\|(\mu(\ta)-1)\na u\|_{L^1_t(\dot B^{0}_{\infty,1})}
+\|\ta\|_{L^1_t(\dot B^{-1}_{\infty,1})}\nonumber\\
\lesssim&\|u_0\|_{\dot B^{-1}_{\infty,1}}+\varepsilon\| u\|_{L^1_t(\dot B^{1}_{\infty,1})}
+\|\na u\|_{L^2_t(L^2)}^2
+\| u\|_{L^\infty_t(L^2)}\|\na u\|_{L^2_t(L^2)}^2\nonumber\\
&+\|\na u\|_{{L}^2_t(L^p)}\|\ta\|_{\widetilde{L}^2_t(\dot B^{4/p}_{p,1})} +\|\mu(\ta)-1\|_{L^\infty} \|u\|_{L^1_t(\dot B^1_{\infty,1})}+\|\ta\|_{L^1_t(\dot B^{-1}_{\infty,1})}\nonumber\\
\lesssim&\|u_0\|_{\dot B^{-1}_{\infty,1}}+\varepsilon\| u\|_{L^1_t(\dot B^{1}_{\infty,1})}
+\|\na u\|_{L^2_t(L^2)}^2
+\| u\|_{L^\infty_t(L^2)}\|\na u\|_{L^2_t(L^2)}^2\nonumber\\
&+\|\na u\|_{\widetilde{L}^2_t(L^p)}(t^{1/2}\|\ta_0\|_{L^2}2^{(1+2/p)N}+
\|\ta\|_{\widetilde{L}^2_t(B_{p,\infty}^{\alpha})}2^{-(\alpha-4/p)N}) \nonumber\\&+\|\mu(\ta)-1\|_{L^\infty} \|u\|_{L^1_t(\dot B^1_{\infty,1})}+t(\|\ta_0\|_{L^q}2^{(2/q-1)L}+2^{-L}\|\ta_0\|_{L^\infty})
,\end{align}
where we have used the following two estimates which can be proved similarly as in \cite{zhangping2017}
\begin{align*}%\label{peinam}
&\|u\cdot\na u\|_{\dot B^{-1}_{\infty,1}}
\lesssim
\|\na u\|_{L^2}^2+
\| u\|_{\dot{H}^{1/2}}\|\na u\|_{L^2}^{1/2}\| u\|_{\dot B^{1}_{\infty,1}}^{1/2},\nonumber\\
&\|(\mu(\ta)-1)\na u\|_{\dot B^{0}_{\infty,1}}
\lesssim\|\na u\|_{L^p}\|\ta\|_{\dot B^{ 4/p}_{p,1}} +\|\mu(\ta)-1\|_{L^\infty} \|u\|_{\dot B^1_{\infty,1}}.
\end{align*}

In the above inequality \eqref{peinam1}, choosing $\varepsilon$ small enough and $L,$ $N$ such that
$$\|\ta_0\|_{L^q}2^{(2/q-1)L}\sim2^{-L}\|\ta_0\|_{L^\infty},\quad
t^{1/2}\|\ta_0\|_{L^2}2^{(1+2/p)N}\sim
\|\ta\|_{\widetilde{L}^2_t(B_{p,\infty}^{\alpha})}2^{-(\alpha-4/p)N},$$
 we can obtain
\begin{align*}%\label{peinam1+7895}
\|u\|_{L^1_t(\dot B^1_{\infty,1})}
\lesssim&\|u_0\|_{\dot B^{-1}_{\infty,1}}
+\|\na u\|_{L^2_t(L^2)}^2
+\| u\|_{L^\infty_t(L^2)}\|\na u\|_{L^2_t(L^2)}^2\nonumber\\
&+\|\na u\|_{{L}^2_t(L^p)}(t^{12}\|\ta_0\|_{L^2})^{\frac{\alpha p-4}{(\alpha+1)p-2}} \|\ta\|_{\widetilde{L}^2_t(B_{p,\infty}^{\alpha})}^{\frac{ p+2}{(\alpha+1)p-2}}+t\|\ta_0\|_{L^q}^{q/2}\|\ta_0\|_{L^\infty}^{1- q/2},
\end{align*}
thus, using \eqref{529e3},  \eqref{jiarui16}, \eqref{pingping2+896}, we have
\begin{align}\label{peinam1+7896}
\|u\|_{L^1_t(\dot B^1_{\infty,1})}
\lesssim&\|u_0\|_{\dot B^{-1}_{\infty,1}}
+E_0^2
(1+E_0)+\mathcal{G}_3(t^{1/2}\|\ta_0\|_{L^2})^{\frac{\alpha p-4}{(\alpha+1)p-2}} \mathcal{G}_4^{\frac{ p+2}{(\alpha+1)p-2}}+t\|\ta_0\|_{L^q}^{ q/2}\|\ta_0\|_{L^\infty}^{1-q/2}.
\end{align}

\end{proof}

With estimate $\| u\|_{L^1_t(\dot B^{1}_{\infty,1})}$ in hand, we can use the following commutator's estimate
 \begin{align*}%\label{}
\|[ \dot{\Delta}_j,u\cdot\nabla]\ta\|_{L^p}\lesssim
 2^{- \alpha j/2}\|\na u\|_{L^\infty}\|\ta\|_{\dot{B}_{p,\infty}^{\alpha/2}}
\end{align*}
which the proof  can be easily obtained  by using Bony's decomposition that
\begin{align}\label{xiaoxiao5}
\|\ta\|_{\widetilde{L}^\infty_t(\dot{B}_{p,\infty}^{\alpha/2})}+
\|\ta\|_{\widetilde{L}^1_t(\dot{B}_{p,\infty}^{3\alpha/2})}\le
\|\ta_0\|_{\dot{B}_{p,\infty}^{\alpha/2}}\exp(C\mathcal{H}(t)).
\end{align}
From estimate \eqref{xiaoxiao5}, we can obtain the following corollary about $\|u\|_{\widetilde{L}^2_t(\dot{B}_{2,\infty}^{3/2})}$:
\begin{corollary}\label{xiaoxiao6}
Let $(\ta, u)$ be a smooth enough solution of \eqref{m} on $[0, T^\ast).$ Then under the assumptions of Theorem \ref{maintheorem} and for any  $t <T^\ast$, there holds
\begin{align*}%\label{xiaoxiao7}
\|u\|_{\widetilde{L}^2_t(\dot{B}_{2,\infty}^{3/2})}\lesssim&
\|u_0\|_{H^1}+E_0\Big(\mathcal{G}_2^{1/2}\big(1+E_0^{1/2}\mathcal{G}_2^{1/4}
+E_0^{1/2}(1+E_0)^{1/2} (1+\mathcal{G}_2)^{1/4}\big)+\exp(C\mathcal{G}_1)\Big)
\triangleq \mathcal{G}_5.
\end{align*}
\end{corollary}
\begin{proof}
From equation \eqref{jiarui5+258}, we can get by a similar derivation of \eqref{peinam1} that
\begin{align}\label{yiming3+896}
\|u\|_{\widetilde{L}^2_t(\dot{B}_{2,\infty}^{3/2})}
\lesssim&
\|u_0\|_{\dot{B}_{2,\infty}^{3/2}}+\|\ta\|_{\widetilde{L}^{4/3}_t(L^2)}
+\|\na u\|_{L^2_t(L^4)}\|u\|_{L^\infty_t(L^2)}
+\|(\mu(\ta)-1)\na u\|_{\widetilde{L}^2_t(\dot{B}_{2,\infty}^{1/2})}
.
\end{align}

By using Bony's decomposition, para-product estimates  and
interpolation inequality, we can obtain for any $p> 4/\alpha$ that
\begin{align*}%\label{}
\|(\mu(\ta)-1)\na u\|_{\widetilde{L}^2_t(\dot{B}_{2,\infty}^{1/2})}
\lesssim&\|\mu(\ta)-1\|_{L^\infty_t(L^\infty)}\|u\|_{\widetilde{L}^2_t(\dot{B}_{2,\infty}^{3/2})}
+\|\theta\|_{\widetilde{L}^\infty_t(\dot{B}_{p,\infty}^{\alpha/2})}
\|\na u\|_{{\widetilde{L}}^2_t(\dot{B}_{\frac{2p}{p-2},2}^{{(1-\alpha)}/{2}})}
\nonumber\\
\lesssim&\|\mu(\ta)-1\|_{L^\infty_t(L^\infty)}\|u\|_{\widetilde{L}^2_t(\dot{B}_{2,\infty}^{3/2})}
+\|\theta\|_{\widetilde{L}^\infty_t(\dot{B}_{p,\infty}^{\alpha/2})}
\|\na u\|_{{L}^2_t(\dot{B}_{2,2}^{0})}^{\alpha+{4}/{p}}
\|\na u\|_{{\widetilde{L}}^2_t(\dot{B}_{2,\infty}^{1/2})}^{1-\alpha+{4}/{p}}
\nonumber\\
\lesssim&\|\mu(\ta)-1\|_{L^\infty_t(L^\infty)}\|u\|_{\widetilde{L}^2_t(\dot{B}_{2,\infty}^{3/2})}
+\|\theta\|_{\widetilde{L}^\infty_t(\dot{B}_{p,\infty}^{\alpha/2})}
\|\na u\|_{{L}^2_t({L}^2)}^{\alpha+{4}/{p}}
\| u\|_{{\widetilde{L}}^2_t(\dot{B}_{2,\infty}^{3/2})}^{1-\alpha+{4}/{p}}
\nonumber\\
\lesssim&\|\mu(\ta)-1\|_{L^\infty_t(L^\infty)}\|u\|_{\widetilde{L}^2_t(\dot{B}_{2,\infty}^{3/2})}
+\varepsilon
\| u\|_{{\widetilde{L}}^2_t(\dot{B}_{2,\infty}^{3/2})}
+\|\theta\|_{\widetilde{L}^\infty_t(\dot{B}_{p,\infty}^{\alpha/2})}
^{{p}/{(\alpha p-4)}}
\|\na u\|_{{L}^2_t({L}^2)}.
\end{align*}
Taking \eqref{tiaojian}  into consideration in the above estimate, we have
\begin{align}\label{yaochen}
\|(\mu(\ta)-1)\na u\|_{\widetilde{L}^2_t(\dot{B}_{2,\infty}^{1/2})}
\le\varepsilon
\| u\|_{{\widetilde{L}}^2_t(\dot{B}_{2,\infty}^{3/2})}
+\|\theta\|_{\widetilde{L}^\infty_t(\dot{B}_{p,\infty}^{\alpha/2})}
^{{p}/{(\alpha p-4)}}
\|\na u\|_{{L}^2_t({L}^2)}.
\end{align}

Substituting \eqref{yaochen}
 into \eqref{yiming3+896} and choosing $\varepsilon$ small enough imply
\begin{align}\label{yiming3+896+456}
\|u\|_{\widetilde{L}^2_t(\dot{B}_{2,\infty}^{3/2})}
\lesssim&
\|u_0\|_{\dot{B}_{2,\infty}^{3/2}}+\|\ta\|_{\widetilde{L}^{4/3}_t(L^2)}
+\|\na u\|_{L^2_t(L^4)}\|u\|_{L^\infty_t(L^2)}
+\|\theta\|_{\widetilde{L}^\infty_t(\dot{B}_{p,\infty}^{\alpha/2})}
^{{p}/{(\alpha p-4)}}
\|\na u\|_{{L}^2_t({L}^2)}
.
\end{align}
On one hand,
from estimates \eqref{pingping31} and  \eqref{jiarui15} we have
\begin{align}\label{xiaoxiao10+896}
\|\na u\|_{L^2_t(L^4)}\lesssim (1+E_0)E_0^{1/2} (1+\mathcal{G}_2)^{1/2}.
\end{align}
On the other hand,
it's easy to get from \eqref{529ee} that
\begin{align}\label{xiaoxiao9}
\|\ta\|_{L^{4/3}_t(L^2)}\lesssim E_0.
\end{align}
Thus, taking estimates \eqref{xiaoxiao10+896}, \eqref{xiaoxiao9} into \eqref{yiming3+896+456} and using \eqref{529e3}, \eqref{jiarui14},  we have
\begin{align*}%\label{yiming3+896+896}
\|u\|_{\widetilde{L}^2_t(\dot{B}_{2,\infty}^{3/2})}
\lesssim&
\|u_0\|_{\dot{B}_{2,\infty}^{3/2}}+E_0
+(1+E_0)E_0^{1/2} (1+\mathcal{G}_2)^{1/2}E_0
+\|\theta\|_{\widetilde{L}^\infty_t(\dot{B}_{p,\infty}^{\alpha/2})}
^{{p}/{(\alpha p-4)}}
E_0\nonumber\\
\lesssim&
\|u_0\|_{H^1}+E_0
+(1+E_0)E_0^{1/2} (1+\mathcal{G}_2)^{1/2}E_0
+\exp\Big(({{p}/{2(\alpha p-4)}})\mathcal{G}_1\Big)
E_0.
\end{align*}
Consequently, we complete the proof of this corollary.
\end{proof}

\section{\Large\bf Proof of  Theorem \ref{maintheorem}}

\subsection{ \large\bf The existence of  Theorem \ref{maintheorem}}
Before giving the existence of Theorem \ref{maintheorem}, we will present the following lemma  about the propagation of low regularities for the temperature function $\ta.$
\begin{lemma}\label{xiaoxiao11}
Let $( u, \ta)$ be a smooth solution of system \eqref{m} on $[0, T^\ast).$ Then under the assumptions of Theorem \ref{maintheorem}, we have for any $t<T^\ast$
\begin{align}\label{xiaoxiao12}
\|\ta\|_{\widetilde{L}^\infty_t(\dot{H}^{-s_0})}\le CE_0(1+E_0(1+E_0+\mathcal{G}_1))\triangleq \mathcal{G}_6
\end{align}
for $E_0$ and $\mathcal{G}_1$ given by \eqref{e0} and \eqref{jiarui14} respectively.
\end{lemma}

\begin{proof}
We first get by a similar derivation of \eqref{xinger} that
\begin{align}\label{feixing}
\|\dot{\Delta}_j\ta(t)\|_{L^2}\le e^{-ct2^{j\alpha}}\|\dot{\Delta}_j\ta_0\|_{L^2}
+C\int_0^te^{-c(t-t')2^{j\alpha}}
\big\|[\dot{\Delta}_j,u\cdot\na]\theta\big\|_{L^2} \,dt'.
\end{align}
To continue our argument, we will use the following commutator's estimate which the proof can be obtained as Lemma 3.3 in \cite{zhangping2017}:
$$\sum_{j\in\Z}2^{-js}\|[\dot{\Delta}_j,u\cdot\na]\theta\|_{L^2}
\le C\|\ta\|_{\dot{H}^{1-s}}\|\na u\|_{L^2}, \quad  -1<s<2.$$
Thus, a simple computation helps us get from \eqref{feixing} and the above estimate that
\begin{align}\label{sunqin1}
\|\ta\|_{\widetilde{L}^\infty_t(\dot{H}^{-s_0})}
\le &
\|\ta_0\|_{\dot{H}^{-s_0}}+C\int_0^t\sum_{j\in\Z}2^{-js_0}\|[\dot{\Delta}_j,u\cdot\na]\theta\|_{L^2}\,dt'
\nonumber\\
\le &
\|\ta_0\|_{\dot{H}^{-s_0}}+C\|\ta\|_{L^2_t(\dot{H}^{1-s_0})}\|\na u\|_{L^2_t(L^2)}
.
\end{align}
By the same manner, we have
\begin{align}\label{sunqin2}
\|\ta\|_{L^2_t(\dot{H}^{1-s_0})}
\le &
\|\ta_0\|_{\dot{H}^{1-s_0-\alpha/2}}+C\int_0^t\sum_{j\in\Z}2^{j(1-s_0-\alpha/2)}\|[\dot{\Delta}_j,u\cdot\na]\theta\|_{L^2}\,dt'
\nonumber\\
\le &
\|\ta_0\|_{\dot{H}^{1-s_0-\alpha/2}}+C\|\ta\|_{L^2_t(\dot{H}^{2-s_0-\alpha/2})}\|\na u\|_{L^2_t(L^2)}
\nonumber\\
\le &
\|\ta_0\|_{\dot{H}^{1-s_0-\alpha/2}}+C
(\|\ta_0\|_{\dot{H}^{2-s_0- \alpha}}+\|\ta\|_{L^2_t(\dot{H}^{3-s_0- \alpha})})\|\na u\|_{L^2_t(L^2)}
.
\end{align}
As $2/3<\alpha\le 1$ and $3-2\alpha<s_0<{4\alpha }/{q}-8\alpha +6$, one can infer $0< {3-s_0- \alpha}\le \alpha$, thus,
\begin{align}\label{sunqin3}
&\|\ta_0\|_{\dot{H}^{1-s_0-\alpha/2}}\le C\|\ta_0\|_{\dot{H}^{-s_0}\cap L^2},\quad
\|\ta_0\|_{\dot{H}^{2-s_0- \alpha}}\le C\|\ta_0\|_{\dot{H}^{-s_0}\cap {H}^{\alpha/2}},
\end{align}
and
\begin{align}\label{sunqin3+1}
\|\ta\|_{L^2_t(\dot{H}^{3-s_0- \alpha})}\le C\|\ta\|_{L^2_t(L^2)\cap L^2_t(\dot{H}^{\alpha}) }.
\end{align}

Inserting the estimates \eqref{sunqin3}, \eqref{sunqin3+1} into \eqref{sunqin2}, we can deduce from \eqref{sunqin1}  that
\begin{align}\label{sunqin4}
\|\ta\|_{\widetilde{L}^\infty_t(\dot{H}^{-s_0})}
\le &\|\ta_0\|_{\dot{H}^{-s_0}}\nonumber\\
&+(\|\ta_0\|_{\dot{H}^{-s_0}\cap L^2}+
(\|\ta_0\|_{\dot{H}^{-s_0}\cap {H}^{\alpha/2}}+\|\ta\|_{L^2_t(L^2)\cap L^2_t(\dot{H}^{\alpha}) })\|\na u\|_{L^2_t(L^2)})\|\na u\|_{L^2_t(L^2)}.
\end{align}
From decay estimate \eqref{529ee} and estimates \eqref{529e3}, \eqref{jiarui14}, we have
$$
\|\na u\|_{L^2_t(L^2)})\le C E_0,\quad \|\ta\|_{L^2_t(L^2)\cap L^2_t(\dot{H}^{\alpha}) }
\le C E_0+\mathcal{G}_1,$$
thus, taking the above estimates into \eqref{sunqin4}, we can finally arrive at \eqref{xiaoxiao12}.
\end{proof}

We are in a position to prove the
existence part of Theorem \ref{maintheorem}.
The strategy first is to
solve an appropriate approximate of \eqref{m} and then prove the uniform bounds for such approximate solutions, and the last step consists in proving the convergence of such approximate solutions to a solution of the original system.
One can check similar argument from page 1239 to page 1240 of \cite{zhangping2017} for details, here, we omit it.

  \bigskip

\subsection{ \large\bf The uniqueness of  Theorem \ref{maintheorem}}
In this subsection, we will present the uniqueness of  Theorem \ref{maintheorem}. As the $\ta$ equation has a supercritical regularity, thus, there will be more complicated discussion than \cite{zhangping2017}.
Let $u^i,\ta^i$ (with $i=1,2$) be two solutions of the system \eqref{m} which satisfy \eqref{zhengzexing1}, \eqref{zhengzexing2}.

Denote
$ (\delta u,\delta \ta,\nabla\delta \Pi)\triangleq(u^2-u^1,\ta^2-\ta^1,\nabla \Pi^2-\nabla \Pi^1).$
Then $(\delta u,\delta \ta,\nabla\delta \Pi)$ solves
\begin{eqnarray}\label{shiping3}
\left\{\begin{aligned}
&\partial_t\delta\theta+u^2\cdot\nabla\delta\theta+ |D|^\alpha \delta\theta =-\delta u\cdot\nabla\theta^1,\\
&\partial_t \delta u+ u^2\cdot\nabla \delta u-\diverg(\mu(\ta^2)d( \delta u))+\nabla\delta\Pi-\delta\theta e_2\\
&\hspace{6cm} =\diverg((\mu(\ta^2)-\mu(\ta^1)) d(u^1))-\delta u\cdot\nabla  u^1,\\
&\diverg \delta u =0,\\
&(\delta \theta,\delta u)|_{t=0}=(0,0).
\end{aligned}\right.
\end{eqnarray}

Taking $L^2$ inner product $\delta u$ with the $\delta u$ equation, $\delta \ta$ with the $\delta \ta$ equation in the above equation, using the H$\mathrm{\ddot{o}}$lder inequality and Young inequality, we can finally get that
\begin{align}\label{shiping5}
\frac12\frac{d}{dt}(\|\delta u\|_{L^2}^2+\|\delta \ta\|_{L^2}^2)+\|\nabla \delta u\|_{L^2}^2
+\|\delta \ta\|_{\dot{H}^{\alpha/2}}^2
\lesssim& \|\delta \ta\|_{L^{{4p}/{(4+4p-3\alpha p)}}}\|\na u^1\|_{L^{{4p}/{(3\alpha p-4-2p)}}}\|\nabla \delta u\|_{L^2}\nonumber\\
&+(1+\|\na \ta^1\|_{L^\infty}
+\|\na u^1\|_{L^2}^2)(\|\delta u\|_{}^2+\|\delta \ta\|_{L^2}^2)\nonumber\\
\lesssim& \varepsilon\|\nabla \delta u\|_{L^2}^2
+
\|\delta \ta\|_{L^{{4p}/{(4+4p-3\alpha p)}}}^2\|\na u^1\|_{L^{{4p}/{(3\alpha p-4-2p)}}}^2\nonumber\\
&+(1+\|\na \ta^1\|_{L^\infty}+\|\na u^1\|_{L^2}^2)(\|\delta u\|_{}^2+\|\delta \ta\|_{L^2}^2).
\end{align}
In the following,
 we will use the following  lemma of which the proof can be obtained similarly to Proposition 3.1 in \cite{zhangping2017} (with a small modification) to control the term $\|\delta \ta\|_{L^{{4p}/{(4+4p-3\alpha p)}}}^2$ in \eqref{shiping5}.
\begin{lemma}\label{yangyi1}
  Denote $\gamma\triangleq{{4p}/({4+4p-3\alpha p}}).$ Assume $p>{4}/{(3\alpha-2)}$ and $\ta_0\in\dot{B}_{\gamma,\infty}^{0}$,
let $v\in L^2_T(\dot{W}^{1,4/\alpha})$ be a solenoidal vector field and $f\in{L^2_t(\dot{B}_{\gamma,\infty}^{-\alpha/2})}$.
Then the equation below
$$\partial_t \ta+ u\cdot\nabla \ta+|D|^\alpha\ta=f\quad and \quad \ta|_{t=0}=\ta_0,$$
has a unique solution $\ta$ so that for $t \le T$
\begin{align}\label{yangyi2}
\|\ta\|_{L^\infty_t(\dot{B}_{\gamma,\infty}^{0})}
\lesssim&\big(
\|\ta_0\|_{\dot{B}_{\gamma,\infty}^{0}}
+\|f\|_{L^2_t(\dot{B}_{\gamma,\infty}^{-\alpha/2})}
\big)\exp(C\|\na v\|_{L^2_t(L^{4/\alpha})}).
\end{align}

\end{lemma}

Applying the first equation in \eqref{shiping3} to the above Lemma \ref{yangyi1} yields
\begin{align}\label{yangyi3}
\|\delta \ta\|_{L^\infty_t(\dot{B}_{\gamma,\infty}^{0})}^2
\le&C
\|\delta u\cdot\nabla\theta^1\|_{L^2_t(\dot{B}_{\gamma,\infty}^{-\alpha/2})} ^2\exp(C\|\na u^2\|_{L^2_t(L^{4/\alpha})})\nonumber\\
\le&
C\exp(C\|\na u^2\|_{L^2_t(L^{4/\alpha})})
\int_0^t\|\delta u\|_{L^2}^2\|\theta^1\|_{\dot{B}_{p,\infty}^{\alpha}}^2
\,dt',
\end{align}
where the following estimate  has been used, which the proof will be given later
\begin{align}\label{yangyi5}
\|\delta u\cdot\nabla\theta^1\|_{L^2_t(\dot{B}_{\gamma,\infty}^{-\alpha/2})} \le&
C\|\delta u\|_{L^2}\|\theta^1\|_{\dot{B}_{p,\infty}^{\alpha}}.
\end{align}
Notice that for any positive integer $N$ and $p>{4}/{(3\alpha-2)},$ we have
\begin{align}\label{yangyi6}
\|\delta \ta\|_{L^{\gamma}}
\lesssim\|\delta \ta\|_{\dot{B}_{\gamma,2}^{0}}
\lesssim\|\delta \ta\|_{L^2}+\sqrt{N}\|\delta \ta\|_{\dot{B}_{\gamma,2}^{0}}
+2^{-N(\alpha/2+{2}/{\gamma}-1)}\|\delta \ta\|_{\dot{H}^{\alpha/2}},
\end{align}
taking $N$ in the above inequality such that
$$
2^{N(\alpha/2+{2}/{\gamma}-1)}\sim
\|\delta \ta\|_{\dot{H}^{\alpha/2}}\big/\|\delta \ta\|_{\dot{B}_{\gamma,\infty}^{0}},$$
then we have
\begin{align}\label{yangyi8}
\|\delta \ta\|_{L^{{4p}/{(4+4p-3\alpha p)}}}
=\|\delta \ta\|_{L^{\gamma}}
\lesssim\|\delta \ta\|_{L^2}+\|\delta \ta\|_{\dot{B}_{\gamma,\infty}^{0}}\ln^{1/2}
\Big(e+(\| \ta_1\|_{\dot{H}^{\alpha/2}}+\| \ta_2\|_{\dot{H}^{\alpha/2}})\big/\|\delta \ta\|_{\dot{B}_{\gamma,\infty}^{0}}\Big)
.
\end{align}
Taking \eqref{yangyi8} into \eqref{shiping5} and choosing $\varepsilon$ small enough, we have
\begin{align}\label{yangyi9}
&\|\delta u\|_{L^\infty_t(L^2)}^2+\|\delta \ta\|_{L^\infty_t(L^2)}^2)+\|\delta \ta\|_{L^\infty_t(\dot{B}_{\gamma,\infty}^{0})}^2+\|\nabla \delta u\|_{L^2_t(L^2)}^2
+\|\delta \ta\|_{L^2_t(\dot{H}^{\alpha/2})}^2\nonumber\\
&\quad\lesssim C\exp(C\|\na u^2\|_{L^2_t(L^{4/\alpha})})
\int_0^t\|\delta u\|_{L^2}^2\|\theta^1\|_{\dot{B}_{p,\infty}^{\alpha}}^2
\,dt'\nonumber\\
&\quad\quad+\int_0^t(1+\|\na \ta^1\|_{L^\infty}+\|\na u^1\|_{L^2}^2+\|\na u^1\|_{L^{{4p}/{(3\alpha p-4-2p)}}}^2)(\|\delta u\|_{}^2+\|\delta \ta\|_{L^2}^2)\,dt'\nonumber\\
&\quad\quad+\int_0^t\|\na u^1\|_{L^{{4p}/{(3\alpha p-4-2p)}}}^2\|\delta \ta\|_{\dot{B}_{\gamma,\infty}^{0}}^{2}\ln
\Big(e+(\| \ta_1\|_{\dot{H}^{\alpha/2}}+\| \ta_2\|_{\dot{H}^{\alpha/2}})\big/\|\delta \ta\|_{\dot{B}_{\gamma,\infty}^{0}}\Big)\,dt'.
\end{align}
Denote
$$Y(t)\triangleq\|\delta u\|_{L^\infty_t(L^2)}^2+\|\delta \ta\|_{L^\infty_t(L^2)}^2+\|\delta \ta\|_{L^\infty_t(\dot{B}_{\gamma,\infty}^{0})}^2.$$
One can deduce from \eqref{yangyi9} that
\begin{align}\label{yangyi10}
Y(t)\lesssim& C\exp(C\|\na u^2\|_{L^2_t(L^{4/\alpha})})
\int_0^tY(t')\|\theta^1\|_{\dot{B}_{p,\infty}^{\alpha}}^2
\,dt'\nonumber\\
&+C\int_0^t(1+\|\na \ta^1\|_{L^\infty}+\|\na u^1\|_{L^2}^2+\|\na u^1\|_{L^{{4p}/{(3\alpha p-4-2p)}}}^2)Y(t')\,dt'\nonumber\\
&+C\int_0^t\|\na u^1\|_{L^{{4p}/{(3\alpha p-4-2p)}}}^2Y(t')\ln
\Big(e+(\| \ta_1\|_{\dot{H}^{\alpha/2}}+\| \ta_2\|_{\dot{H}^{\alpha/2}})\big/Y(t')\Big)\,dt'.
\end{align}
For any
${8}/{(3\alpha-2)}\le p<{1}/{C\|\mu(\cdot)-1\|_{L^\infty}}$ and $2/3<\alpha\le 1$, by using the interpolation inequality, we have
\begin{align*}%\label{lingzhi1}
&\|\na \ta^1\|_{L^1_t(L^\infty)}\lesssim\| \ta^1\|_{\widetilde{L}^\infty_t(\dot{B}_{p,\infty}^{\alpha/2})}
^{\frac{3\alpha p-2p-4}{2\alpha p}}
\|\ta^1\|_{\widetilde{L}^1_t(\dot{B}_{p,\infty}^{3\alpha/2})}
^{\frac{2p-\alpha p+4}{2\alpha p}},
\end{align*}
\begin{align*}%\label{lingzhi2}
\|\na u^2\|_{L^2_t(L^{4/\alpha})}\lesssim\|\na u^2\|_{L^2_t(L^{2})}^{\frac{\alpha p-4)}{2(p-2)}}\|\na u^2\|_{L^2_t(L^{p})}^{\frac{p(2-\alpha)}{2(p-2)}},
\quad
\|\na u^1\|_{L^2_t(L^{{4p}/{(3\alpha p-4-2p)}})}^2
\lesssim\|\na u^1\|_{L^2_t(L^{2})}^\frac{3\alpha p-2p-8}{ p-2}\|\na u^1\|_{L^2_t(L^{p})}^\frac{4p+4-3\alpha p}{ p-2}.
\end{align*}
Thus, applying Osgood's Lemma \ref{osgood} to \eqref{yangyi10}, we can infer that
$Y(t)=0.$
This complete the uniqueness of Theorem \ref{maintheorem}.

Consequently, we have completed the proof of our's main Theorem \ref{maintheorem}.

\bigskip
\subsection*{\large\bf Acknowledgements} This work are supported by NSFC under grant number 11601533 and  the Postdoctoral Science Foundation
of China under  grant number 2016M592560.

\end{document}